\documentclass{article}
\usepackage{amsmath}
\usepackage{amsfonts}
\usepackage{amsthm}
\usepackage{amssymb, setspace}
\usepackage{mathtools}
\usepackage{lipsum}
\usepackage{color,graphicx,epsfig,geometry,fancyhdr,hyperref}
\usepackage[T1]{fontenc}
\usepackage{float}
\usepackage{subfig}
\usepackage{commath}
\usepackage{bbm}
\usepackage{mathrsfs,fleqn}
\usepackage{pgfplots}
\pgfplotsset{compat=newest}
\newtheorem{theorem}{Theorem}[section]

\newtheorem{lemma}{Lemma}[section]
\newtheorem{definition}{Definition}[section]

\newcommand{\Z}{\mathbb{Z}}

\newcommand{\R}{\mathbb{R}}
\newcommand{\C}{\mathbb{C}}

\usepackage{comment}
\usepackage{ulem}

\graphicspath{{paper-pics/}}

\begin{document}

\begin{flushleft}
\Large 
\noindent{\bf \Large Analysis of the monotonicity method for an anisotropic scatterer with a conductive boundary}
\end{flushleft}

\vspace{0.2in}
{\bf  \large Victor Hughes, Isaac Harris, and Heejin Lee}\\
\indent {\small Department of Mathematics, Purdue University, West Lafayette, IN 47907 }\\
\indent {\small Email: \texttt{vhughes@purdue.edu}, \texttt{harri814@purdue.edu}, and  \texttt{lee4485@purdue.edu}}\\

\begin{abstract}
\noindent In this paper, we consider the inverse scattering problem associated with an anisotropic medium with a conductive boundary. We will assume that the corresponding far--field pattern is known/measured and we consider two inverse problems. First, we show that the far--field data uniquely determines the boundary coefficient. Next, since it is known that anisotropic coefficients are not uniquely determined by this data we will develop a qualitative method to recover the scatterer. To this end, we study the so--called monotonicity method applied to this inverse shape problem. This method has recently been applied to some inverse scattering problems but this is the first time it has been applied to an anisotropic scatterer. This method allows one to recover the scatterer by considering the eigenvalues of an operator associated with the far--field operator. We present some simple numerical reconstructions to illustrate our theory in two dimensions. For our reconstructions, we need to compute the adjoint of the Herglotz wave function as an operator mapping into $H^1$ of a small ball. 
\end{abstract}

\noindent {\bf Keywords}: Inverse Scattering $\cdot$ Monotonicity Method $\cdot$ Far--Field Operator $\cdot$ Anisotropic Scatterer \\

\noindent {\bf MSC}: 35J05, 35J25

\section{Introduction} \label{acbc_section_intro}
Here, we will study the {\it inverse shape problem} of recovering the scatterer as well as the {\it inverse boundary parameter problem} from given far--field data. For the inverse shape problem, we analyze the monotonicity method for reconstructing an anisotropic medium with a conductive boundary. The monotonicity method is a qualitative reconstruction scheme for the unknown region that is comparable to the so--called linear sampling \cite{LSM-maxwell-book,MUSIC-LSM} and factorization methods \cite{factor-music,kirschbook}. This is do to the fact that each method derives an imaging function that requires the eigenvalue(or singular) value decomposition of an operator associated with the known/measured far--field operator. The monotonicity method does not require us to have any restriction on the wave number i.e. valid even at  transmission eigenvalues(see for e.g. \cite{te-cbc,On-the-interior-TE,aniso-te-cbc,otheraniso-tecbc}). This is not true for the linear sampling and factorization methods but here we need to assume we know the sign of the anisotropic contrast. Our main contribution will be to study the application of the monotonicity method for an anisotropic material with a conductive boundary on an unbounded domain. This problem has not been studied before and will require original analysis to complete. Indeed, one of the main calculations needed to apply the monotonicity method to an anisotropic scatterer is to compute the adjoint of the Herglotz wave function as an operator mapping into $L^2(\mathbb{S}^{d-1}) \to H^1\big( B(z, \epsilon) \big)$. Here $B(z, \epsilon)$ denotes the `sampling ball' at the `sampling point' $z \in \R^d$. The main result of this paper states that for a known operator associated with the measurements only has finitely many negative eigenvalues provided that the sampling ball is contained in the scatterer. This gives us a way to recover the scatterer via computing the eigenvalues of a known operator with little {\it a priori} information on the geometry. Since the imaging function is given by computing the eigenvalues of known operators, this would imply stability by the continuity of the spectrum.

Here the conductive boundary condition is modeled by a Robin condition that states the normal/co-normal derivative of the total field has a jump across the boundary of the scatterer that is proportional to the total field. We remark that this is different from the model presented in \cite{fm-anisocbc} but was recently studied in \cite{aniso-te-cbc,otheraniso-tecbc}. The aforementioned papers consider the associated transmission eigenvalue problem. We wish to fill the current gap in knowledge i.e. to study the associated inverse problems for this model. For this model, we will need to derive a factorization of the far--field operator. See \cite{fm-anisoKL} for a a factorization of the far--field operator with no conductivity parameter. We note that our analysis is still valid for the case when there is no conductive boundary parameter. This is then used in conjunction with the main result in \cite{mono_method_paper} to derive a monotonicity method. In \cite{mono_method_paper} the authors give a theoretical blueprint for applying the monotonicity method to inverse scattering problems. This method has been studied in \cite{mono-invscat-media} for other inverse media scattering problems as well as in \cite{mono_AG} for some inverse obstacle scattering problems. The monotonicity method has been used in shape reconstruction for the $p$--Laplace \cite{mono-p-laplace}, electrical impedance tomography \cite{mono-eit} as well as other models/configurations see for e.g. \cite{mono-waveguide,mono-eit-xtreme,mono-invscat-crack,mono-schrodinger,mono-other}. This paper adds to the growing literature on the monotonicity method applied to many imaging modalities.

The rest of the paper is organized as follows: In Section \ref{acbc_section_formulation}, we formulate the direct and inverse scattering problem for an anisotropic material with a conductive boundary. Next, in  Section \ref{acbc_section_uniq} we study the uniqueness of the inverse parameter problem for determining the boundary parameter. We are able to prove that the far--field pattern uniquely determines the boundary parameter. In section \ref{acbc_section_factorization}, we decompose the far--field operator into a symmetric factorization, this is necessary for the use of the monotonicity method. Once an an appropriate factorization is derived, in Section \ref{acbc_section_monomethod} we provide an explanation on how to apply the monotonicity method to this problem given our factorization of the far field operator using the theory developed in \cite{mono_method_paper}. Finally, in Section \ref{acbc_section_numerics}, we provide numerical examples to validate our theoretical results.

\section{Formulation of the Problem} \label{acbc_section_formulation}
We start by discussing the direct scattering problem for an anisotropic material with a conductive boundary. Let $D \subset \mathbb{R}^d$ for $d=$ 2 or 3 be a simply connected open set with Lipschitz boundary $\partial D$, where $\nu$ represents the unit outward normal vector. The region $D$ denotes the scatterer that we illuminate with an incident plane wave $u^i:=\text{e} ^{\text{i} kx\cdot\hat{y}}$ such that $k>0$ is the wave number and $\hat{y} \in \mathbb{S}^{d-1}$ denotes the incident direction. We assume that we have a symmetric matrix-valued function $A(x) \in \mathscr{C}^1 (\overline{D}, \mathbb{C}^{d\times d})$ that is uniformly positive definite in $D$ satisfying 
\begin{center}
    $ \overline{\xi} \cdot \text{Re}(A(x))\xi \geq A_{\text{min}}|\xi|^2 $ \quad and \quad $\overline{\xi} \cdot \text{Im}(A(x))\xi \leq 0$ \quad for a.e. $x \in D$ and $\xi \in \mathbb{C}^d$
\end{center}
where $A_{\text{min}}$ is a positive constant. We also assume that the refractive index $n(x) \in L^\infty (\mathbb{R}^d)$ satisfies
\begin{center}
    $\text{Im}(n(x)) \geq 0$ \quad for a.e. $x \in D$.
\end{center}
We assume that the parameters $A(x)$ and $n(x)$ for an anisotropic material satisfy
$$A(x)=I \quad \text{and} \quad n(x)=1\quad \text{for }x \in \R^d \setminus \overline{D}. $$
We also consider the conductivity parameter $\eta \in L^\infty(\partial D)$, with the condition
\begin{center}
    $\text{Im}(\eta(x)) \geq 0$ \quad for a.e. $x \in \partial D$.
\end{center}

We have that the direct scattering problem for an anisotropic material with a conductive boundary condition is formulated as follows: find $u^s \in H^1_{loc}(\mathbb{R}^d)$ such that
\begin{align}
 \Delta u^s + k^2u^s = 0 \quad \text{in $\mathbb{R}^d \setminus \overline{D}$} \quad \text{ and } \quad  \nabla \cdot A\nabla u+k^2nu=0 \quad &\text{in $D$} \label{directacbc1} \\ 
(u^s+u^i)^+ = u^-  \quad \text{ and } \quad  \partial_\nu (u^s+u^i)^+ = \nu \cdot A\nabla u^- - \eta u \quad &\text{on $\partial D$}  \label{directacbc2} 
\end{align}
where $u = u^s+u^i$ denotes the total field. Here the superscripts $+$ and $-$ demonstrate approaching the boundary from the outside and inside of $D$, respectively. To close the system, we assume that the scattered field satisfies the radiation condition 
\begin{align}
    \partial_r u^s -\text{i} ku^s = \mathcal{O}\Big(\frac{1}{r^{(d+1)/2}}\Big)  \quad \text{ as } \quad r \to \infty \label{Sommerfield_acbc}
\end{align}
uniformly with respect to $\hat{x}=\frac{x}{r}$ where $r=|x|$. With the above assumptions on the coefficients we can prove well-posedness of \eqref{directacbc1}--\eqref{directacbc2} along with the radiation condition. This can be done in a similar manner to the analysis in \cite{fmconductbc} where a variational formulation of \eqref{directacbc1}--\eqref{directacbc2} for the scattered field $u^s$ is used to study the system. 

We are interested in recovering the scatterer $D$ and/or the boundary parameter $\eta$ from measured far--field data. To this end, recall that the asymptotic expansion of the scattering field $u^s(\cdot \, , \hat{y})$ has the form
\begin{align*}
    u^s(x,\hat{y}) = \gamma_d\frac{\text{e}^{\text{i}k|x|}}{|x|^{\frac{d-1}{2}}} \Bigg\{  u^\infty(\hat{x}, \hat{y}) + \mathcal{O}\Bigg( \frac{1}{|x|} \Bigg)   \Bigg\}
\end{align*}
where $\gamma_d$ is given by 
$$\gamma_2 = \frac{\text{e}^{\text{i}\pi /4}}{\sqrt{8\pi k}}\quad \text{and} \quad \gamma_3=\frac{1}{4\pi}.$$ 
Here we let $\hat{x} = x/|x|$ and $u^\infty(\hat{x},\hat{y})$ is the corresponding far--field pattern associated with \eqref{directacbc1}--\eqref{directacbc2}, which is dependent on the incident direction $\hat{y}$ and the observation direction $\hat{x}$. 
Due to the fact that $u^s$ is a radiating solution to the Helmholtz equation in $\R^d \setminus \overline{D}$ we have that it can be written using Green's representation formula \cite{CCH-book}. Therefore, we have that the far--field pattern has the following integral representation
\begin{align*}
    u^\infty(\hat{x},\hat{y}) = -\int_{\partial \Omega} \text{e}^{-\text{i}k\hat{x}\cdot z}\partial_\nu u^s - u^s\partial_\nu \text{e}^{-\text{i}k\hat{x}\cdot z} \, \mathrm{d}s(z)
\end{align*}
for any region $\Omega$ such that $\overline{D} \subseteq \Omega$. With the far--field pattern, we now define the far--field operator $F: L^2(\mathbb{S}^{d-1}) \longrightarrow L^2(\mathbb{S}^{d-1})$ as
\begin{align*}
    (Fg)(\hat{x}) := \int_{\mathbb{S}^{d-1}} u^\infty (\hat{x}, \hat{y})g(\hat{y}) \,  \text{d}s(\hat{y}) \quad \text{ where } g \in L^2(\mathbb{S}^{d-1}).
\end{align*}
It is well known that the far--field operator $F$ is compact. 

We are interested in the inverse problem of recovering $D$ from the given the far--field data. The main goal is to apply the monotonicity method to perform this reconstruction of $D$, given we have the far--field data. Another inverse problem that we will consider is the uniqueness for recovering the boundary parameter $\eta$ from the known far--field data.

\section{Uniqueness of the Inverse Parameter Problem} \label{acbc_section_uniq}
In this section, we show that the far--field pattern uniquely determines the boundary parameter $\eta$ provided that $D$ is known. In general, the anisotropic material parameter $A$ is not uniquely recovered by the far--field pattern(see for e.g. \cite{GyCo}). Therefore, we will consider a qualitative method to recover $D$ in the preceding section. Here we will prove a uniqueness result for the inverse boundary parameter problem. We are considering the injectivity of the mapping 
$$\eta \in L^\infty(\partial D)  \longmapsto u^\infty(\hat{x}, \hat{y}) \in L^2 \big( \mathbb{S}^{d-1} \times \mathbb{S}^{d-1}\big).$$
To this end, in addition to the assumption on $\eta$ in Section \ref{acbc_section_formulation}, we assume that
\begin{align}\label{assump_k}
\varphi \in H^1_0 (D) \quad \text{satisfies} \quad \nabla \cdot A\nabla \varphi+k^2n \varphi=0 \,\, \text{in $D$} \quad \text{implies $\varphi=0\,\,$ a.e. in $D$ }
\end{align}
and we need a density result as in \cite{Heejin1,Haddar1}. Notice, that \eqref{assump_k} is true for all $k \in \R$ provided that $A$ and $n$ are complex--valued. 
\begin{lemma}\label{uniq_density}
Under the assumption \eqref{assump_k}, the set
\begin{align*}
\mathcal{U} \coloneqq \left\{ u(\cdot, \hat{y})|_{\partial D} : u \textrm{ is the total field of  \eqref{directacbc1}--\eqref{directacbc2} for } \hat{y} \in \mathbb{S}^{d-1}\right\}
\end{align*}
is dense in $L^2(\partial D)$.
\end{lemma}

\begin{proof}
To prove the claim, we will show that $\mathcal{U}^\perp$ is trivial. For any $\phi \in \mathcal{U}^\perp$ and
let $p \in H^1_{loc}(\mathbb{R}^d)$ be the unique radiating solution of 
\begin{align*}
 \Delta p + k^2p = 0 \quad \text{in $\mathbb{R}^d \setminus \overline{D}$} \quad \text{ and } \quad  \nabla \cdot A\nabla p+k^2np=0 \quad &\text{in $D$}  \\ 
p^+ = p^-  \quad \text{ and } \quad  \partial_\nu p^+ - \nu \cdot A\nabla p^- +\eta p = \phi \quad &\text{on $\partial D$}.
\end{align*}
If $u$ is the total field of \eqref{directacbc1}--\eqref{directacbc2}, from the boundary conditions, we have that
\begin{align*}
\int_{\partial D} p \partial_\nu u^+ - u \partial_\nu p^+ \, \mathrm{d}s 
= \int_{\partial D} p (\nu \cdot A\nabla u^-) - u(\nu \cdot A\nabla p^-) - u\phi \, \mathrm{d}s.
\end{align*}
Notice that from Green's 2nd identity and the symmetry of $A$, 
\begin{align*}
 \int_{\partial D} p (\nu \cdot A\nabla u^-) - u(\nu \cdot A\nabla p^-) \mathrm{d}s
 = \int_D \nabla p \cdot A \nabla u - \nabla u \cdot A \nabla p \, \mathrm{d}x = 0.
\end{align*}
Therefore, combining the above two equations, we have that
\begin{align*}
\int_{\partial D} p \partial_\nu u^+ - u \partial_\nu p^+ \, \mathrm{d}s  = - \int_{\partial D} u\phi \, \mathrm{d}s = 0,
\end{align*}
that is,
\begin{align*}
\int_{\partial D} p \partial_\nu u^i - u^i \partial_\nu p^+ \, \mathrm{d}s  
=   \int_{\partial D}  u^s \partial_\nu p^+ - p \partial_\nu u^s  \, \mathrm{d}s .
\end{align*}
Applying Green's 2nd identity to $u^s$ and $p$ in $B_R \setminus \overline{D}$ where $B_R$ is the ball of radius $R>0$ centered at the origin and containing $D$ in its interior, we have
\begin{align*}
 \int_{\partial D}  u^s \partial_\nu p^+ - p \partial_\nu u^s  \, \mathrm{d}s 
  = - \int_{B_R}  u^s \partial_\nu p - p \partial_\nu u^s  \, \mathrm{d}s
 \rightarrow 0 \,\, \text{ as }  \,\, R \rightarrow \infty,
 \end{align*}
 since $u^s$ and $p$ are satisfy the radiation condition \eqref{Sommerfield_acbc}. Thus,
 \begin{align}\label{uniq_greenidentity}
 \int_{\partial D} p \partial_\nu u^i - u^i \partial_\nu p^+ \, \mathrm{d}s =0.
 \end{align}
 Recall, that the fundamental solution $\Phi(x,y)$ of the Helmholtz equation is given by  
 \begin{align}\label{fundamentalsolHE}
 \Phi(x,y) =\frac{\text{i} }{4} H^{(1)}_0(k|x-y|) \text{ for } d=2  \quad \text{and} \quad  \Phi(x,y) =\frac{\text{e}^{\text{i}k|x-y|}}{4\pi|x-y|}  \text{ for } d=3,
 \end{align}
  where $H_0^{(1)}$ is the Hankel function of the first kind and of order $0$. Now, consider the incident field $u^i(x, -\hat{y}) = \text{e}^{-\text{i}kx\cdot \hat{y}}$ which is the far field pattern of $\Phi(y,x)$.
 From the Green's representation formula(see for e.g. \cite{CCH-book}),
 \begin{align*}
 p(y) = \int_{\partial D} p \partial_\nu \Phi(\cdot, y) - \partial_\nu p^+ \Phi(\cdot, y) \, \mathrm{d}s \quad \text{ for } \,\,\, y \in \R^d \setminus \overline{D}. 
  \end{align*}
This implies that left-hand side of \eqref{uniq_greenidentity} represents the far--field pattern $p^\infty (\hat{y})$ of the function $p(y)$. Since $p^\infty(\hat{y}) = 0$ from \eqref{uniq_greenidentity}, we obtain that $p(y)$ vanishes in $\mathbb{R}^d \setminus \overline{D}$ by Rellich's lemma(see for e.g. \cite{CCH-book}). This implies that $p^- = 0$ on $\partial D$. By assumption \eqref{assump_k}, we conclude that $p \equiv 0$ in $D$ which completes the proof.
\end{proof}

We showed the density of the set $\mathcal{U}$ in $L^2(\partial D)$. With this, we will show that the far--field pattern uniquely determines the boundary parameter under the same assumption as in Lemma \ref{uniq_density}. 
\begin{theorem}
Under the assumption \eqref{assump_k}, the knowledge of the far--field pattern $u^\infty(\hat{x}, \hat{y})$  for all $\hat{x}$ and $\hat{y} \in \mathbb{S}^{d-1}$ uniquely determines the boundary parameter $\eta \in L^\infty(\partial D)$.
\end{theorem}
\begin{proof}
Let us denote $u_j (\cdot \,  , \hat{y})$ as the total field to \eqref{directacbc1}--\eqref{directacbc2} with the boundary parameter $\eta$ replaced by $\eta_j$ for each $j=1,2$, and let $u_j^s$ be the corresponding scattered field. Assuming that the far--field patterns of $u_1^\infty (\hat{x}, \hat{y}) =u_2^\infty (\hat{x}, \hat{y})$ for all $\hat{x}$ and $\hat{y} \in \mathbb{S}^{d-1}$. Then by Rellich's lemma we have that $u_1^s ({x}, \hat{y}) = u_2^s({x}, \hat{y})$ for all $x \in \mathbb{R}^d \setminus \overline{D}$ and $\hat{y} \in \mathbb{S}^{d-1}$, which implies that $u_1 (\cdot \, , \hat{y})= u_2 (\cdot \, , \hat{y})$ on $\partial D$. With assumption \eqref{assump_k}, $u_1  (\cdot \, , \hat{y})$ and $u_2  (\cdot \, , \hat{y})$ coincide in $\mathbb{R}^d$.

 Now, let us denote $u_1$ and $u_2$ by $u$. From the boundary condition,
\begin{align*}
0 = \partial_\nu u^+ - \nu \cdot A \nabla u^{-} + \eta_1 u = \partial_\nu u^+ - \nu \cdot A \nabla u^{-} + \eta_2 u \text{ on } \partial D.
\end{align*}
Then, $(\eta_1 - \eta_2) u (\cdot \, , \hat{y}) = 0$ on $\partial D$. Since $(\eta_1 - \eta_2) \in L^\infty(\partial D) \subset L^2(\partial D)$, it follows from Lemma \ref{uniq_density} that $\eta_1 = \eta_2$ a.e. on $\partial D$.
\end{proof}

With this result, we know that the inverse boundary parameter problem is unique with respect to the data. In the following section we will solve the inverse shape problem i.e. recover the unknown scatterer $D$ from the given far--field pattern/operator. 

\section{Factorization of Far Field Operator} \label{acbc_section_factorization}
In order to utilize the monotonicity method for reconstructing the region $D$, we will first need to obtain a symmetric factorization of the far--field operator. To this end, motivated by the original direct scattering problem (\ref{directacbc1})--(\ref{directacbc2}), we consider the problem of finding $u \in H^1_{loc}(\mathbb{R}^d)$ for any given $v \in H^1(D)$ such that
\begin{align}
      \nabla \cdot A\nabla u + k^2nu = \nabla \cdot (I-A)\nabla v + k^2(1-n)v \quad &\text{ in } \mathbb{R}^d  \setminus \partial D \label{reformdirectacbc1} \\ 
      u^+-u^- = 0  \quad \text{ and } \quad  \partial_\nu u^+ - \nu \cdot A\nabla u^- = \nu \cdot (A-I)\nabla v  - \eta (u+v) \quad &\text{ on } \partial D \label{reformdirectacbc2} \\
      \partial_r u -\text{i} ku = \mathcal{O}\Big(\frac{1}{r^{(d+1)/2}}\Big)  \quad &\text{ as } r \to \infty. \nonumber
\end{align}
Note that $I-A$ and $1-n$ are supported in $D$. This system is equivalent to (\ref{directacbc1})--(\ref{directacbc2}) when you use the fact the incident field solves the Helmholtz equation in all of $\R^d$. The variational formulation for the above system is given by
\begin{align}
   -\int_{B_R} A\nabla u \cdot\nabla\overline{\phi} - k^2nu\overline{\phi} \,\text{d}x &+ \int_{\partial B_R} \overline{\phi}\Lambda_k u \,\text{d}s \nonumber \\ 
   & = \int_D \nabla \overline{\phi} \cdot (A-I)\nabla v - k^2 (n-1)v\overline{\phi} \,\text{d}x - \int_{\partial D} \overline{\phi}\eta (u+v) \,\text{d}s \label{varformdirectacbc}
\end{align}
for all $\phi \in H^1(B_R)$, where $B_R$ is the open ball of radius $R$ such that $\overline{D} \subset B_R$. Here we let  
$$\Lambda_k: H^{1/2}(B_R) \longrightarrow H^{-1/2}(B_R)$$ 
denote the Dirichlet to Neumann map on $\partial B_R$ for the radiating solution to the Helmholtz equation on the exterior of $B_R$, defined as 
\begin{equation*}
    \Lambda_k f = \partial_\nu \varphi \big|_{\partial B_R} \quad \text{ where }\quad  \Delta \varphi+k^2 \varphi=0 \text{ in $\mathbb{R}^d \setminus \overline{B_R}$} \quad \text{ and } \quad \varphi\big|_{\partial B_R} = f, 
\end{equation*}
along with the radiation condition (\ref{Sommerfield_acbc}). From Theorem 5.22 of \cite{Cakoni-Colton-book} we know that $\Lambda_k$ is a bounded linear operator. This is a direct consequence of the well-posedness of the above Dirichlet problem along with the (Neumann) Trace Theorem.

To get an initial factorization of $F$ we first define the source to far--field pattern operator
\begin{align}
G: H^1(D) \longrightarrow L^2(\mathbb{S}^{d-1}) \quad \text{given by} \quad Gv = u^\infty \label{Gdefine}
\end{align}
where $u \in H^1_{loc}(\mathbb{R}^d)$ is the unique solution to \eqref{reformdirectacbc1}--\eqref{reformdirectacbc2} for a given $v \in H^1(D)$. Next we define the bounded linear operator 
\begin{align}
H: L^2(\mathbb{S}^{d-1}) \longrightarrow H^1(D) \quad \text{given by}\quad (Hg)(x) = \int_{\mathbb{S}^{d-1}} \text{e}^{\text{i}kx\cdot \hat{y}}g(\hat{y}) \, \text{d}s(\hat{y}) \Big|_D . \label{Hdefine}
\end{align}
By the superposition principle, we have that the far--field operator associated with (\ref{directacbc1})--(\ref{directacbc2}) is given by $F=GH$. In order to use the monotonicity method, we need to have a symmetric factorization of far--field operator. To achieve this, we compute the adjoint of the operator $H$.
\begin{theorem}\label{Hstardefine}
    The operator $H^*: H^1(D) \longrightarrow  L^2(\mathbb{S}^{d-1})$ is given by 
    $$H^*f=-w^\infty \quad \text{for all } \quad f \in H^1(D)$$
    where $w^\infty$ is the far--field pattern for the unique radiating solution $w$ that satisfies 
\begin{align}
        -\int_{B_R} \nabla w \cdot \nabla \overline{\phi} - k^2w\overline{\phi} \, \mathrm{d}x + \int_{\partial B_R} \overline{\phi}\Lambda_kw \, \mathrm{d}s = (f,\phi)_{H^1(D)}  \label{acbc_Hstar_varform}  
    \end{align}
    for all $\phi \in H^1(B_R)$.
\end{theorem}
\begin{proof}
    To prove the claim, we let $\psi = Hg$ be extended for all $x \in \mathbb{R}^d$. Then we have that
    \begin{align*}
        (H^*f, g)_{L^2(\mathbb{S}^{d-1})} = (f, Hg)_{H^1(D)} = -\int_{B_R} \nabla w \cdot \nabla \overline{\psi} - k^2w\overline{\psi} \, \mathrm{d}x + \int_{\partial B_R} \overline{\psi}\partial_\nu w \, \mathrm{d}s.
    \end{align*}
    Using Green's 1st Identity and along with the fact that $\psi$ is a solution to the Helmholtz equation in $\mathbb{R}^d$ we now obtain
    \begin{align*}
        (H^*f, g)_{L^2(\mathbb{S}^{d-1})} &= \int_{B_R} w\overline{(\Delta \psi + k^2 \psi)} \, \mathrm{d}x + \int_{\partial B_R} \overline{\psi}\partial_\nu w - w\partial_\nu \overline{\psi} \, \mathrm{d}s \\
        &= \int_{\partial B_R} \overline{\psi}\partial_\nu w - w\partial_\nu \overline{\psi} \, \mathrm{d}s \\
        &= \int_{\mathbb{S}^{d-1}}\overline{g(\hat{y})} \Bigg[ \int_{\partial B_R} \text{e}^{-\text{i}kx\cdot \hat{y}}\partial_\nu w - w\partial_\nu \text{e}^{-\text{i}kx\cdot \hat{y}} \, \mathrm{d}s(x) \Bigg] \, \mathrm{d}s(\hat{y}) \\
        &= -\int_{\mathbb{S}^{d-1}}\overline{g(\hat{y})} w^\infty(\hat{y}) \, \mathrm{d}s(\hat{y}) = (-w^\infty, g)_{L^2(\mathbb{S}^{d-1})}.
    \end{align*}
    This implies that $H^*f=-w^\infty$, proving the claim.
\end{proof}
Now that we have the adjoint of our operator $H$, we proceed as is done in \cite{fm-shixu,near_field} to show that $G=-H^*T$, for some operator $T$. To this end, from (\ref{reformdirectacbc1}) we have that 
$$ \Delta u + k^2u = \nabla \cdot (I-A)\nabla (u+v) + k^2(1-n)(u+v) \quad \text{ in } \mathbb{R}^d \setminus \partial D $$
where $u \in H^1_{loc}(\mathbb{R}^d)$ is the solution of (\ref{reformdirectacbc1})--(\ref{reformdirectacbc2}) for any given $v \in H^1(D)$. The corresponding variational form is given by
\begin{align}
    -\int_{B_R} \nabla u \cdot\nabla\overline{\phi} &- k^2u\overline{\phi} \, \mathrm{d}x  + \int_{\partial B_R} \overline{\phi}\Lambda_k u \, \mathrm{d}s \nonumber \\ 
    &= \int_D \nabla \overline{\phi} \cdot (A-I)\nabla (u+v) - k^2 (n-1)(u+v)\overline{\phi} \, \mathrm{d}x - \int_{\partial D} \overline{\phi}\eta (u+v) \, \mathrm{d}s \label{varform2directacbc}
\end{align}
for any $\phi$. By wellposedness of the direct scattering problem, we have that $v \longmapsto u$ is a bounded linear mapping. With this fact we can use the Riesz Representation Theorem on the right hand side of the variational form (\ref{varform2directacbc}) which implies that there exists a bounded linear operator $T: H^1(D) \longrightarrow H^1(D)$ such that for all $v \in H^1(D)$ 
\begin{align}
    (Tv, \phi)_{H^1(D)} = \int_D \nabla \overline{\phi} \cdot (A-I)\nabla (u+v) - k^2 (n-1)(u+v)\overline{\phi} \, \mathrm{d}x - \int_{\partial D} \overline{\phi}\eta (u+v) \, \mathrm{d}s. \label{acbcTdef}
\end{align}
It is clear by (\ref{varform2directacbc}), that we have
\begin{align*}
 -\int_{B_R} \nabla u \cdot\nabla\overline{\phi} - k^2u\overline{\phi} \, \mathrm{d}x + \int_{\partial B_R} \overline{\phi}\Lambda_k u \, \mathrm{d}s =   (Tv, \phi)_{H^1(D)}.
\end{align*}
By appealing to the definition of $G$ and $H^*$ as defined in \eqref{Gdefine} and \eqref{acbc_Hstar_varform} we have that 
$$u^\infty = Gv \quad \text{ and } \quad u^\infty = -H^*Tv \quad \text{ for any $v \in H^1(D)$ }$$ 
which implying $G=-H^*T$. Now, recalling that we have the initial factorization $F=GH$, which implies the desired symmetric factorization $F=-H^*TH$.
\begin{lemma}
    The far-field operator $F:L^2(\mathbb{S}^{d-1}) \longrightarrow L^2(\mathbb{S}^{d-1})$ for the scattering problem (\ref{directacbc1})--(\ref{directacbc2}) has the factorization $F=-H^*TH$.
\end{lemma}

With this result we have the factorization required to apply the theory of the monotonicity method to our inverse problem. In the next section we will develop the required theoretical results. This will require us to study the analytical properties of the operators $H$ and $T$ that are used to factorize the far-field operator. 


\section{The Monotonicity Method} \label{acbc_section_monomethod}
In this section, we connect the support of the scattering object $D$ to the spectrum of an operator associated with the far--field operator $F$. We recall the symmetric factorization of the far--field operator $F=-H^*TH$, this factorization is important for applying the theory of the monotonicity method. We now provide the following definition that will give some background for our analysis going forward. 

\begin{definition}
    Let $A$, $B$ be self-adjoint compact operators on a Hilbert space. We write
    $$ A \leq_{\mathrm{fin}} B$$
    If $B-A$ had finitely many negative eigenvalues.
\end{definition}
The main result of the monotonicity method that we will employ here is given in Theorem 4.2 of \cite{mono_method_paper}.  This result outlines how the number of negative eigenvalues of an operator associated with the far-field operator will increase provided that the sampling point is in the interior of the scatterer $D$. For completeness we state a version of Theorem 4.2 in \cite{mono_method_paper} for our setting. 

\begin{theorem}\label{mono-result}
(Theorem 4.2 of \cite{mono_method_paper}) Assume we have compact operators $\mathcal{F}_j :\mathbb{Y} \rightarrow \mathbb{Y}$ acting on the Hilbert space $\mathbb{Y}$ with factorizations 
$$\mathcal{F}_j = H^*_j T_j H_j \quad \text{ such that } \quad H_j : \mathbb{Y} \to \mathbb{X}_j \quad \text{ and } \quad T_j : \mathbb{X}_j \to \mathbb{X}_j$$
with $T_j $ and $H_j$ being bounded linear operators with Hilbert spaces $\mathbb{X}_j$ for $j=1$ and 2.   
\begin{enumerate}
\item Assume that 
\begin{enumerate}
\item $\mathrm{Re}(T_1)$ is the sum of a positive coercive operator and self-adjoint compact operator.
\item There exists a compact operator $R : \mathbb{X}_1 \to \mathbb{X}_2$ such that $H_2 = R H_1$.
\end{enumerate} 
$$\text{Then we have that $ \mathrm{Re}(\mathcal{F}_2) \leq_{\mathrm{fin}} \mathrm{Re}(\mathcal{F}_1). $}$$\\

\item Assume that
\begin{enumerate}
\item $\mathrm{Re}(T_2)$ is the sum of a positive coercive operator and self-adjoint compact operator.
\item $\mathrm{Ran}(H_1^*) \cap \mathrm{Ran}(H_2^*) = \{ 0 \}$ where Dim$\big(\mathrm{Ran}(H_2^*)\big)=\infty$. 
\end{enumerate} 
$$\text{Then we have that $ \mathrm{Re}(\mathcal{F}_2) \nleq_{\mathrm{fin}} \mathrm{Re}(\mathcal{F}_1). $}$$
\end{enumerate} 
\end{theorem}

We note that the real part of a compact operator $\mathcal{F}: \mathbb{Y} \rightarrow \mathbb{Y}$ is given by 
$$\mathrm{Re}(\mathcal{F}) = \frac{1}{2} \left( \mathcal{F} + \mathcal{F}^*\right).$$
Notice, that this is always a self-adjoint compact operator. To apply the above result to our problem we first notice we already have the needed factorization $F=-H^*TH$. The next thing to accomplish is to show that $\pm \mathrm{Re}(T)$ is the sum of a positive coercive operator and self-adjoint compact operator where $T$ is defined by \eqref{acbcTdef}. 

Going forward, we want to prove that we can apply the monotonicity method result under the analytical assumption that 
\begin{align*}
    \text{Re}(A)-I>0 \quad \text{uniformly in } D
\end{align*}
or for some constant $\alpha >0$
\begin{align*}
    I-\text{Re}(A)-\alpha \text{Im}(A)>0 \quad \text{uniformly in }D \,\, \textrm{ and} \quad \text{Re}(A)-\frac{1}{\alpha} \text{Im}(A)\geq 0 \quad \text{in }D.
\end{align*}
The notation $M>0$ ($M\geq 0)$ denotes that the matrix $M$ is a positive definite (semi-positive definite) matrix. With this in mind we will prove that the operator $T$ has the desired properties to apply the above result. 

Recall, the definition of the operator $T: H^1(D) \longrightarrow H^1(D)$ for any $v \in H^1(D)$ 
\begin{align*}
    (Tv, \phi)_{H^1(D)} = \int_D \nabla \overline{\phi} \cdot (A-I)\nabla (u+v) - k^2 (n-1)(u+v)\overline{\phi} \,\text{d}x - \int_{\partial D} \overline{\phi}\eta (u+v) \,\text{d}s 
\end{align*}
where $u \in H^1_{loc}(\mathbb{R}^d)$ satisfies the auxiliary problem (\ref{reformdirectacbc1})--(\ref{reformdirectacbc2}) with $v \in H^1(D)$ given. Therefore, we have that for any $ v_j$ there exists a unique $u_j$ that satisfies the variational formulation (\ref{varform2directacbc}) for $j$=1 and 2. By definition we have that 
\begin{align*}
    (Tv_1, v_2)_{H^1(D)} &= \int_D \nabla \overline{v_2} \cdot (A-I)\nabla (u_1+v_1) - k^2 (n-1)(u_1+v_1)\overline{v_2} \,\text{d}x - \int_{\partial D} \overline{v_2}\eta (u_1+v_1) \,\text{d}s \\
    &= \int_D \nabla \overline{(u_2+v_2)} \cdot (A-I)\nabla (u_1+v_1) - k^2 (n-1)(u_1+v_1)\overline{(u_2+v_2)} \,\text{d}x   \\
    &\hspace{1in} - \int_D \nabla \overline{u_2} \cdot (A-I)\nabla (u_1+v_1) - k^2 (n-1)(u_1+v_1)\overline{u_2} \,\text{d}x  \\
    &\hspace{1in} - \int_{\partial D} \overline{(u_2+v_2)}\eta (u_1+v_1) \,\text{d}s+ \int_{\partial D} \overline{u_2}\eta (u_1+v_1) \,\text{d}s 
\end{align*}
Now using the variational form (\ref{varform2directacbc}) corresponding to $(u_1,v_1)$ with $\phi=u_2$ we have 
\begin{align}
-\int_{B_R} \nabla u_1 \cdot\nabla\overline{u_2} - k^2u_1\overline{u_2} \,\text{d}x &+ \int_{\partial B_R} \overline{u_2}\Lambda_ku_1 \,\text{d}s \nonumber \\ 
    &\hspace{-1in}= \int_D \nabla \overline{u_2} \cdot (A-I)\nabla (u_1+v_1) - k^2 (n-1)(u_1+v_1)\overline{u_2} \,\text{d}x - \int_{\partial D} \overline{u_2}\eta (u_1+v_1) \,\text{d}s \nonumber
\end{align}
Plugging this into the expression for $T$ we get
\begin{align*}
(Tv_1, v_2)_{H^1(D)} &= \int_D \nabla \overline{(u_2+v_2)} \cdot (A-I)\nabla (u_1+v_1) - k^2 (n-1)(u_1+v_1)\overline{(u_2+v_2)} \,\text{d}x \\
    &- \int_{\partial D} \eta (u_1+v_1)\overline{(u_2+v_2)} \,\text{d}s + \int_{B_R} \nabla u_1 \cdot\nabla\overline{u_2} - k^2u_1\overline{u_2} \,\text{d}x - \int_{\partial B_R} \overline{u_2}\Lambda_k u_1 \,\text{d}s
\end{align*}
With these calculations, we show that our operator $T$ has the desired decomposition for us to use the monotonicity method given by Theorem \ref{mono-result}. 

\begin{theorem}\label{T=coercive+compact}
    Let $T: H^1(D) \longrightarrow H^1(D)$ be as defined in \eqref{acbcTdef}. 
    \begin{enumerate}
        \item If $\mathrm{Re}(A)-I>0$ uniformly in $D$, then $\mathrm{Re}(T)$ is the sum of a positive coercive operator and self-adjoint compact operator.
        \item If $I-\mathrm{Re}(A)-\alpha \mathrm{Im}(A)>0$ uniformly in $D$ and $\mathrm{Re}(A)-\frac{1}{\alpha}\mathrm{Im}(A) \geq 0$ in $D$, then $-\mathrm{Re}(T)$ is the sum of a positive coercive operator and self-adjoint compact operator.
    \end{enumerate}
\end{theorem}
\begin{proof}
(1) Using Riesz Representation Theorem on the above calculations for the operator $T$, we can define the following bounded linear operators $S$, $K:H^1(D)\longrightarrow H^1(D)$ as follows
\begin{align}
    (Sv_1,v_2)_{H^1(D)} &= \int_D \nabla \overline{(u_2+v_2)} \cdot (A-I)\nabla (u_1+v_1) + (u_1+v_1)\overline{(u_2+v_2)} \,\text{d}x \nonumber \\
    &\hspace{0.5in}+ \int_{B_R} \nabla u_1 \cdot\nabla\overline{u_2} + u_1\overline{u_2} \,\text{d}x - \int_{\partial B_R} \overline{u_2}\Lambda_k u_1 \,\text{d}s \label{acbcSpart1}
\end{align}
and
\begin{align}
    (Kv_1,v_2)_{H^1(D)} &= - \int_D k^2(n-1)(u_1+v_1)\overline{(u_2+v_2)} \,\text{d}x - \int_D (u_1+v_1)\overline{(u_2+v_2)} \,\text{d}x \nonumber \\
    &\hspace{0.5in}- \int_{\partial D} \eta (u_1+v_1)\overline{(u_2+v_2)} \,\text{d}s + \int_{B_R} (k^2+1)u_1\overline{u_2} \,\text{d}x. \label{acbcKpart1}
\end{align}
With this it is clear that $T=S+K$ which implies that $\text{Re}(T) = \text{Re}(S)+\text{Re}(K)$. Notice, by virtue of the compact embeddings $H^1(D)$ into $L^2(D)$, $H^1(B_R)$ into $ L^2(B_R)$, and $H^{1/2}(\partial D)$ into $ L^2(\partial D)$(see for e.g. \cite{salsa}), we have that $K$, and thus $\text{Re}(K)$ is also compact operator. It is also clear to see that $\text{Re}(K)$ is a self-adjoint operator. \\
To complete the proof, we will show that $\text{Re}(S)$ is a positive coercive operator provided that $\text{Re}(A)-I>0$ uniformly in $D$. To this end, notice that 
\begin{align}
    (\text{Re}(S)v_1,v_1)_{H^1(D)} &= \int_D \nabla \overline{(u_1+v_1)} \cdot (\text{Re}(A)-I)\nabla (u_1+v_1) + |u_1+v_1|^2 \,\text{d}x \nonumber \\
    &\hspace{0.5in}+ \int_{B_R} |\nabla u_1|^2 + |u_1|^2 \,\text{d}x - \int_{\partial B_R} \overline{u_1} \text{Re}(\Lambda_k)u_1 \,\text{d}s \label{acbcRe(S)part1}
\end{align}
where $\text{Re}(\Lambda_k)$ is a non-positive operator (see for e.g. \cite{CCH-book}). We want to show that $\text{Re}(S)$ is positive coercive i.e. there is a $\beta>0$ independent of $v_1 \in H^1(D)$ such that 
$$ (\text{Re}(S)v_1,v_1)_{H^1(D)} \geq \beta ||v_1||^2_{H^1(D)}.$$
To this end, we proceed by considering a contradiction argument. We assume $\text{Re}(S)$ has no such constant, then there exists a sequence $v^n \in H^1(D)$ with corresponding $u^n \in H^1_{loc}(\mathbb{R}^d)$ satisfying (\ref{reformdirectacbc1})--(\ref{reformdirectacbc2}) such that 
$$||v^n||_{H^1(D)}=1 \quad \text{and} \quad (\text{Re}(S)v^n , v^n)_{H^1(D)} \leq \frac{1}{n} .$$
Since $\text{Re}(A)-I$ is positive definite and $\text{Re}(\Lambda_k)$ is non-positive, as $n \rightarrow \infty$, we have 
$$(\text{Re}(S)v^n,v^n)_{H^1(D)} \rightarrow 0 \quad \text{implying} \quad u^n \rightarrow 0 \quad \text{and} \quad v^n \rightarrow 0  \quad \text{in } H^1(D) $$
This contradicts the normalization of $v^n$ in $H^1(D)$. Therefore, $\text{Re}(S)$ is a positive coercive operator. This proves the claim for $\text{Re}(A)-I>0$  uniformly in $D$. \\

(2) Since we have new assumptions on the coefficients, we will have to derive a new variational expression of the operator $T$. To this end, recall that by \eqref{acbcTdef} for a given $v_j \in H^1(D)$ with corresponding $u_j \in H^1_{loc}(\mathbb{R}^d)$ satisfying (\ref{reformdirectacbc1})--(\ref{reformdirectacbc2}) we have that
\begin{align}
    (Tv_1, v_2)_{H^1(D)} &= \int_D \nabla \overline{v_2} \cdot (A-I)\nabla (u_1+v_1) - k^2 (n-1)(u_1+v_1)\overline{v_2} \,\text{d}x - \int_{\partial D} \eta (u_1+v_1) \overline{v_2} \,\text{d}s \nonumber \\
    &= \int_D \nabla \overline{v_2} \cdot (A-I)\nabla v_1 - k^2 (n-1)v_1\overline{v_2} \,\text{d}x - \int_{\partial D} \eta (u_1+v_1) \overline{v_2} \,\text{d}s \nonumber \\
    &\hspace{1.5in} + \int_D \nabla \overline{v_2} \cdot (A-I)\nabla u_1 - k^2 (n-1)u_1\overline{v_2} \,\text{d}x. \nonumber
\end{align}
In an effort to rewrite the above expression for $T$, we use the variational form (\ref{varformdirectacbc}) with respect to $(v_2,u_2)$ and let $\phi = u_1$, obtaining
\begin{align*}
-\int_{B_R} A\nabla u_2 \cdot\nabla\overline{u_1} &- k^2nu_2\overline{u_1} \,\text{d}x + \int_{\partial B_R} \overline{u_1}\Lambda_ku_2 \,\text{d}s \\ 
&= \int_D \nabla \overline{u_1} \cdot (A-I)\nabla v_2 - k^2 (n-1)v_2\overline{u_1} \,\text{d}x - \int_{\partial D} \eta (u_2+v_2)\overline{u_1} \,\text{d}s 
\end{align*}
By taking the conjugate of the above equation, we can rewrite the definition of $-T$ we obtain
\begin{align*}
    (-Tv_1,v_2)_{H^1(D)} &=  \int_D \nabla \overline{v_2} \cdot (I-A)\nabla v_1 + k^2 (n-1)v_1\overline{v_2} \,\text{d}x + \int_{\partial D} \eta v_1 \overline{v_2} \,\text{d}s  \\
    &\hspace{0.5in}- \int_D \nabla \overline{v_2} \cdot (A-\overline{A})\nabla u_1 - k^2 (n-\overline{n})u_1\overline{v_2} \,\text{d}x + \int_{\partial D} (\eta - \overline{\eta}) u_1 \overline{v_2} \,\text{d}s \\
    &\hspace{0.5in}+  \int_{B_R} \nabla \overline{u_2} \cdot \overline{A}\nabla u_1 - k^2\overline{n}u_1\overline{u_2} \,\text{d}x - \int_{\partial D} \overline{\eta} u_1 \overline{u_2} \,\text{d}s - \int_{\partial B_R} u_1\overline{\Lambda_ku_2} \,\text{d}s.  
\end{align*}
Now, we compute $-(\text{Re}(T)v_1,v_2)_{H^1(D)}= -\frac{1}{2}\big((T+T^*)v_1,v_2 \big)_{H^1(D)}$ we can determine 
$$(T^*v_1,v_2)_{H^1(D)} =\overline{(T v_2, v_1)}_{H^1(D)}$$
by the above calculations to obtain
\begin{align*}
(-\text{Re}(T)v_1,v_2)_{H^1(D)} &=  \int_D \nabla \overline{v_2} \cdot (I-\text{Re}(A))\nabla v_1 + k^2 (\text{Re}(n)-1)v_1\overline{v_2} \,\text{d}x + \int_{\partial D} \text{Re}(\eta) v_1 \overline{v_2} \,\text{d}s  \\
   &+  \int_{B_R} \nabla \overline{u_2} \cdot \text{Re}(A)\nabla u_1 - k^2 \text{Re}(n)u_1\overline{u_2} \,\text{d}x - \int_{\partial D} \text{Re}(\eta) u_1 \overline{u_2} \,\text{d}s \\
    &- \frac{1}{2}\int_{\partial B_R} \overline{u_2}\Lambda_k u_1 + u_1\overline{\Lambda_ku_2} \, \text{d}s  \\
    &- \text{i}\int_D \nabla \overline{v_2} \cdot \text{Im}(A)\nabla u_1 - k^2 \text{Im}(n)u_1\overline{v_2} \,\text{d}x + \text{i}\int_{\partial D} \text{Im}(\eta) u_1 \overline{v_2} \,\text{d}s \\
    &+ \text{i}\int_D \nabla \overline{u_2} \cdot \text{Im}(A)\nabla v_1 - k^2 \text{Im}(n)v_1\overline{u_2} \,\text{d}x - \text{i}\int_{\partial D} \text{Im}(\eta) v_1 \overline{u_2} \,\text{d}s.
\end{align*}
Using Riesz Representation Theorem on the definition of the operator $-\text{Re}(T)$, we can define the following bounded linear operators $S,K:H^1(D)\longrightarrow H^1(D)$ as
\begin{align}
(Sv_1,v_2)_{H^1(D)} &=  \int_D \nabla \overline{v_2} \cdot (I-\text{Re}(A))\nabla v_1 + v_1\overline{v_2} \,\text{d}x \nonumber\\
    &+ \int_{B_R} \nabla \overline{u_2} \cdot \text{Re}(A)\nabla u_1 \,\text{d}x - \frac{1}{2}\int_{\partial B_R} \overline{u_2}\Lambda_ku_1 + u_1\overline{\Lambda_ku_2}\,\text{d}s \nonumber \\
    &- \text{i}\int_D \nabla \overline{v_2} \cdot \text{Im}(A)\nabla u_1 \,\text{d}x + \text{i}\int_D \nabla \overline{u_2} \cdot \text{Im}(A)\nabla v_1 \,\text{d}x
\end{align}
and 
$$ (Kv_1,v_2)_{H^1(D)} = (-\text{Re}(T)v_1,v_2)_{H^1(D)} - (Sv_1,v_2)_{H^1(D)} . $$
Notice that $(Kv_1,v_2)_{H^1(D)}$ contains only $L^2$ terms in $D$, $B_R$, and on $\partial D$, therefore the compactness of $K$ follows similarly to the other case. It is also clear by the variational formulation that $K$ is a self-adjoint operator.\\

Now we wish to show that $S$ is a positive coercive operator. To this end, first we use Young's Inequality to get the estimate
$$\Big| \int_D \nabla \overline{v} \cdot \text{Im}(A)\nabla u \,\text{d}x \Big| \leq \frac{\alpha}{2}(\text{Im}(A)\nabla v, \nabla v)_{H^1(D)} + \frac{1}{2\alpha}(\text{Im}(A)\nabla u, \nabla u)_{H^1(D)}$$
for any $\alpha >0$. Therefore we have 
\begin{align*}
        (Sv_1,v_1)_{H^1(D)} &\geq \Big( \big[ I-\text{Re}(A)-\alpha \text{Im}(A) \big]\nabla v_1, \nabla v_1 \Big)_{L^2(D)} + (v_1,v_1)_{L^2(D)} \\
    &+ \Big( \big[ \text{Re}(A)-\frac{1}{\alpha} \text{Im}(A) \big]\nabla u_1, \nabla u_1 \Big)_{H^1(D)} - \int_{\partial B_R} \overline{u_1}\text{Re}(\Lambda_k)u_1 \, \text{d}s \\
    &\geq C||v_1||^2_{H^1(D)}
\end{align*}
        meaning that we have completed this part of the proof provided $\alpha>0$ is a constant such that $I-\text{Re}(A)-\alpha \text{Im}(A)>0$ uniformly in $D$ and $\text{Re}(A)-\frac{1}{\alpha}\text{Im}(A) \geq 0$ in $D$.
\end{proof}

Now, we turn attention to satisfying Theorem \ref{mono-result}. To this end, we need to find another operator with a similar factorization to the far--field operator $F$. With this in mind, being motivated by the work in \cite{mono-invscat-media} we now define $B:=B(z,\epsilon)$ to be the ball of radius $\epsilon >0$ centered at the point $z\in \mathbb{R}^d$, and the operator
$$ H_B : L^2(\mathbb{S}^{d-1}) \longrightarrow L^2(B) \quad \text{ given by } \quad (H_{B} g)(x) = \int_{\mathbb{S}^{d-1}} \text{e}^{\text{i}kx\cdot \hat{y}}g(\hat{y}) \, \text{d}s(\hat{y}) \Big|_B . $$
Notice, that when $B(z,\epsilon) \subset D$ the operator $H_B$ is the restriction of the operator $H$ defined in \eqref{Hdefine} on the ball $B(z,\epsilon) \subset D$ in the $L^2(B)$. This implies that $H_B = R_B H$ where 
\begin{align} \label{R_BDefinition}
R_B: H^1(D) \rightarrow L^2(B) \quad R_B f = f\big|_B \quad \text{provided that $B \subset D$}
\end{align}
is the restriction operator onto $B(z,\epsilon)$. Again, by appealing to the compact embedding of $H^1$ into $L^2$ we have that $R_B$ is a compact operator. 

With this in mind, we now have just about all we need to state the main result. The last piece of the puzzle to apply Theorem \ref{mono-result} to our scattering problem is to consider the intersection of the range of $H^*$ and $H^*_B$. In order to apply the result we need to show that the corresponding ranges are disjoint provided that the `sampling ball' $B=B(z,\epsilon)$ and $D$ are disjoint. This will ultimately allow us to test whether or not the `sampling point' $z\in \mathbb{R}^d$ is in the support of the scatterer $D$. 

\begin{theorem}\label{mono-disjoint-range}
    If for a given sampling point $z\in \mathbb{R}^d$ and radius $\epsilon>0$ we have that $B(z,\epsilon) \cap D = \emptyset$ then $\mathrm{Ran}(H^*_B) \cap \mathrm{Ran}(H^*) = \{0\}$.
\end{theorem}
\begin{proof}
    To begin, we will assume $B(z,\epsilon) \cap D = \emptyset$, and let $g \in \mathrm{Ran}(H^*_B) \cap \mathrm{Ran}(H^*)$. By Theorem \ref{Hstardefine} we have that for some $f_1 \in H^1(D)$ by \eqref{acbc_Hstar_varform} there exists $u_D \in H^1_{loc}(\R^d)$ such that $g=H^*f_1= u_D^\infty$ that satisfy
    $$ \Delta u_D + k^2u_D = 0 \quad \text{in }\mathbb{R}^d\setminus \overline{D}.$$
    Similarly, we have that for some $f_2 \in L^2(B)$ just as in \cite{mono-invscat-media} there exists $u_B \in H^1_{loc}(\R^d)$ such that $g=H_B^*f_2=u_B^\infty$ that satisfy
    $$ \Delta u_B + k^2 u_B = 0 \quad \text{in }\mathbb{R}^d\setminus \overline{B}.$$
     Here both $u_D$ and $u_B$ satisfy the radiation condition \eqref{Sommerfield_acbc}. By Rellich's Lemma, we have that $u_B = u_D$ in $\R^d \setminus \big(\overline{D} \cup \overline{B}\big)$. With this, we can now define $w \in H^1_{loc}(\mathbb{R}^d)$ as
    \begin{equation*}
        w = \begin{cases}
        u_D & \quad \text{in } D, \\
        u_B & \quad \text{in }  \mathbb{R}^d\setminus \overline{D}.
    \end{cases}
    \end{equation*}
    Notice, that $w$ is a radiating solution to Helmholtz equation in all of $\mathbb{R}^d$, and thus $w=0$ in $\mathbb{R}^d$. This implies that $0=w^\infty=u_D^\infty$, proving the claim since $g=u_D^\infty=0$.
\end{proof}

Notice, that by the above results we can apply Theorem \ref{mono-result} to our inverse problem. Indeed, let us consider the case when $\text{Re}(A)-I>0$. Then we would have that in Theorem \ref{mono-result} we consider
$$\mathcal{F}_1 = -F =  H^* T H \quad \text{and} \quad \mathcal{F}_2 = H^*_B H_B.$$ 
By Theorem \ref{T=coercive+compact} gives that $\text{Re}(T)$ is the sum of a positive coercive operator and a self-adjoint compact operator. Recall, that $H_B = R_B H$ and by the definition given in \eqref{R_BDefinition} we have that $R_B$ is compact. Therefore, by appealing to (1) of Theorems \ref{mono-result} and \ref{T=coercive+compact} we have that provided the sampling ball $B(z,\epsilon) \subset D$ then
$$H^*_B H_B \leq_{\mathrm{fin}} - \text{Re}(F).$$
Now, by appealing to (2) of Theorem \ref{mono-result} as well as Theorem \ref{mono-disjoint-range} we have that provided the sampling ball $B(z,\epsilon) \cap D=\emptyset$ then 
$$H^*_B H_B \nleq_{\mathrm{fin}} - \text{Re}(F).$$
Here we use that fact that $H_B$ is injective which implies that $H^*_B$ has a dense range i.e. the range is infinite dimensional. Similarly for the case when $I-\text{Re}(A)-\alpha \text{Im}(A)>0$ uniformly in $D$ and $\text{Re}(A)-\frac{1}{\alpha}\text{Im}(A) \geq 0$ in $D$. For this case we note that in Theorem \ref{mono-result} we let 
$$\mathcal{F}_1 = F =  H^*(- T ) H \quad \text{and} \quad \mathcal{F}_2 = H^*_B H_B.$$ 
This gives the main result of the paper that derives a monotonicity based reconstruction method to recover the unknown scatterer.  

\begin{theorem}\label{aniso-mono-method}
    Let $D \subset \mathbb{R}^d$ be a simply connected open set with a Lipschitz boundary and $B=B(z,\epsilon)$ be the ball of radius $\epsilon$ centered at $z \in \mathbb{R}^d$, we have the following:
    \begin{enumerate}
        \item If $\mathrm{Re}(A)-I>0$, then  
        $$ B(z,\epsilon) \subset D \quad \iff \quad H_B^*H_B \leq_{\mathrm{fin}} -\mathrm{Re}(F).$$
        \item If $I-\mathrm{Re}(A)-\alpha \mathrm{Im}(A)>0$ and $\mathrm{Re}(A)-\frac{1}{\alpha}\mathrm{Im}(A) \geq 0$ for some $\alpha >0$, then 
        $$ B(z,\epsilon) \subset D \quad \iff \quad H_B^*H_B \leq_{\mathrm{fin}} \mathrm{Re}(F). $$
    \end{enumerate}
\end{theorem}

Theorem \ref{aniso-mono-method} implies that the eigenvalues of the self-adjoint operator $ \pm \mathrm{Re}(F)-H_B^*H_B$ can be used to recover $D$. In the next section, we will provide some numerical examples of this to recover circular regions in two dimensions with respect to noisy data.


\section{Numerical Results} \label{acbc_section_numerics}
In previous sections, we have proved that we can use the monotonicity method to perform reconstructions of the scatterer $D$.  Now, we will provide numerical examples to validate our theoretical results. It is worth noting that the monotonicity method given in Theorem \ref{aniso-mono-method} is a qualitative reconstruction method that is valid for all wave number $k$.  Unlike the linear sampling and factorization methods that have to avoid wave numbers that are so--called transmission eigenvalues. Notice, that even though the monotonicity method is theoretically valid for all wave numbers, one does need to have the {\it a priori} information on the matrix--valued coefficient to apply this method.

In order to apply the monotonicity method to recover $D$, we need to compute $H^*_BH_B$ where the sampling ball $B=B(z,\epsilon)$. To this end, in the following subsection we derive an integral representation for the operator $H^*_BH_B$. We will do the calculation in both 2 and 3 dimensions. Recall, that the operator 
$$H^*_BH_B: L^2(\mathbb{S}^{d-1}) \longrightarrow L^2(\mathbb{S}^{d-1})$$
so we will write it as an integral operator with an explicit kernel function.

\subsection{Computing the Integral Representation of $H^*_BH_B$}
In this section, we provide the calculations for writing $H^*_B H_B$ as an integral operator mapping $L^2(\mathbb{S}^{d-1})$ into itself. Note, that here we consider $H_B$ as an operator mapping into $H^1(B)$ which implies that $H^*_B$ is given in Theorem \ref{Hstardefine} and is associated with the variational formulation given in \eqref{acbc_Hstar_varform}. This is due to the fact that in our investigation  the reconstructions are more accurate for this case. Recall, that $H^*_{B}f=-w^\infty$ where $w$ solves \eqref{acbc_Hstar_varform} for any given $f \in H^1(D)$. To this end, we will need to determine the strong form of the PDE associated with the equivalent variational formulation given in  \eqref{acbc_Hstar_varform}.

By appealing to Green's Identities we have that $w\in H^1_{loc}(\mathbb{R}^d)$ satisfying the variational formulation \eqref{acbc_Hstar_varform} will be the solution to 
\begin{align*}
    \Delta w + k^2w = f- \Delta f \quad \text{ in } B=B(z_j, \epsilon) \quad &\text{ and }\quad  \Delta w + k^2w = 0 \quad \text{ in } \mathbb{R}^d \setminus \overline{B} \\
    w^+ - w^- = 0 \quad &\text{ and }\quad \partial_\nu w^+ - \partial_\nu w^- = \partial_\nu f \quad \text{ on } \partial B
\end{align*}
for any $f \in H^1(B)$ such that $\Delta f \in L^2(B)$, closed with the Sommerfeld radiation condition. From this, we get the expression for $H^*_B$, given by
\begin{align*}
    (H^*_{B}) f(\hat{x}) = -w^\infty (\hat{x})= \int_{B} \text{e}^{-\text{i}k\hat{x}\cdot z} (f-\Delta f) \, \text{d}z + \int_{\partial B} \text{e}^{-\text{i}k\hat{x}\cdot z} \partial_\nu f \, \text{d}s(z) 
\end{align*}
just as in \cite{regfm2}. Now, using the fact that $H_{B}g$ is a smooth solution to the Helmholtz equation in $B$ we have that 
\begin{align*}
(H^*_{B}H_{B}) g (\hat{x}) &= \int_{B} \text{e}^{-\text{i}k\hat{x}\cdot z} (H_{B}g - \Delta H_{B}g) \, \text{d}z + \int_{\partial B} \text{e}^{-\text{i}k\hat{x}\cdot z} \partial_\nu H_B g \, \text{d}s(z) \\
&= \int_{B} \text{e}^{-\text{i}k\hat{x}\cdot z} (k^2+1)(H_{B}g) \, \text{d}z + \int_{\partial B} \text{e}^{-\text{i}k\hat{x}\cdot z} \partial_\nu (H_{B} g) \, \text{d}s(z).
\end{align*}
In order to proceed, we will need to evaluate the volume and surface integral above.

With this, we will start by considering the volume above to obtain that 
\begin{align*}
 \int_{B} \text{e}^{-\text{i}k\hat{x}\cdot z} (k^2+1)(H_{B}g) \, \text{d}z &= \int_{B} \text{e}^{-\text{i}k\hat{x}\cdot z} (k^2+1)\Bigg[ \int_{\mathbb{S}^{d-1}} \text{e}^{\text{i}k\hat{y}\cdot z}g(\hat{y}) \, \text{d}s(\hat{y})\Bigg] \, \text{d}z \\
    &= (k^2+1) \int_{\mathbb{S}^{d-1}} \Bigg[ \int_{B} \text{e}^{-\text{i}kz\cdot (\hat{x}-\hat{y})} \, \text{d}z \Bigg] g(\hat{y}) \, \text{d}s(\hat{y})
\end{align*}
Using polar/spherical coordinates to represent the inner integral over $B = B(z_j, \epsilon)$, we obtain that  
\begin{align*}
\int_{B} \text{e}^{-\text{i}kz\cdot (\hat{x}-\hat{y})} \,\text{d}z &= \text{e}^{-\text{i}kz_j\cdot (\hat{x}-\hat{y})} \int_{B(0,\epsilon)} \text{e}^{-\text{i}kz\cdot (\hat{x}-\hat{y})} \, \text{d}z  \\
    &=  \text{e}^{-\text{i}kz_j\cdot (\hat{x}-\hat{y})} \int_{\mathbb{S}^{d-1}} \int_0^\epsilon \text{e}^{-\text{i}kz\cdot (\hat{x}-\hat{y})} r^{d-1} \, \text{d}r \, \text{d}s(\hat{z}) \\
    &= \text{e}^{-\text{i}kz_j\cdot (\hat{x}-\hat{y})}  \int_0^\epsilon \Bigg[ \int_{\mathbb{S}^{d-1}} \text{e}^{-\text{i}k\hat{z}\cdot r(\hat{x}-\hat{y})} \, \text{d}s(\hat{z}) \Bigg] r^{d-1} \, \text{d}r 
\end{align*}
where $z=r\hat{z}$ and $r=|z|$. Recall, the Funk--Hecke integral identities(see for e.g. \cite{nf-harris}) that are given as follows
\begin{align*}
     \int_{\mathbb{S}^{d-1}} \text{e}^{-\text{i}k\hat{z}\cdot r(x-y)} \, \text{d}s(\hat{z}) = \begin{cases}
        2\pi J_0(kr|x-y|) & \quad \text{for } d=2, \\ \\
        4\pi j_0(kr|x-y|) & \quad \text{for } d=3
    \end{cases}
\end{align*}
where $J_m$ are Bessel functions of the first kind of order $m$ and  $j_m$ are Spherical Bessel functions of the first kind of order $m$. Using well known recurrence relationships for Bessel functions we can calculate the following anti--derivates 
$$ \int tJ_0(t)\, \text{d}t = tJ_1(t) \quad \text{and} \quad  \int t^2j_0(t)\, \text{d}t = t^2j_1(t) .$$
Therefore, for $d=2$ we have
\begin{align*}
    \int_{B} \text{e}^{-\text{i}kz\cdot (\hat{x}-\hat{y})} \, \text{d}z &= \text{e}^{-\text{i}kz_j\cdot (\hat{x}-\hat{y})} \int_0^\epsilon J_0(kr|\hat{x}-\hat{y}|) r \, \text{d}r \\
    &= \text{e}^{-\text{i}kz_j\cdot (\hat{x}-\hat{y})} \Bigg( \frac{2\pi\epsilon}{k|\hat{x}-\hat{y}|} J_1(k\epsilon|\hat{x}-\hat{y}|) \Bigg) \\
    &= \text{e}^{-\text{i}kz_j\cdot (\hat{x}-\hat{y})} \pi \epsilon^2 \Bigg( J_0(k\epsilon|\hat{x}-\hat{y}|) + J_2(k \epsilon |\hat{x}-\hat{y}|) \Bigg)
\end{align*}
where we have used the recurrence relationship $(2/t)J_1(t) = J_0(t)+J_2(t)$ along with a substitution in the above calculations. We obtain that the volume integral is given by 
\begin{align*}
    &\int_{B} \text{e}^{-\text{i}k\hat{x}\cdot z} (k^2+1)(H_{B}g) \, \text{d}z \\
    &\hspace{1in}= (k^2+1)\int_{\mathbb{S}^1} \Bigg[ \text{e}^{-\text{i}kz_j\cdot (\hat{x}-\hat{y})} \pi \epsilon^2 \Bigg( J_0(k\epsilon|\hat{x}-\hat{y}|) + J_2(k\epsilon|\hat{x}-\hat{y}|) \Bigg) \Bigg]g(\hat{y})\, \text{d}s(\hat{y}) .
\end{align*}
Similar calculation for $d=3$, we obtain that 
\begin{align*}
    &\int_{B} \text{e}^{-\text{i}k\hat{x}\cdot z} (k^2+1)(H_{B}g) \, \text{d}z \\
    &\hspace{1in}= (k^2+1)\int_{\mathbb{S}^2} \Bigg[ \text{e}^{-\text{i}kz_j\cdot (\hat{x}-\hat{y})} \frac{4}{3}\pi \epsilon^3 \Bigg( j_0(k\epsilon|\hat{x}-\hat{y}|) + j_2(k\epsilon|\hat{x}-\hat{y}|) \Bigg) \Bigg]g(\hat{y})\, \text{d}s(\hat{y}) 
\end{align*}
by appealing to the recurrence relationship $(3/t)j_1(t) = j_0(t)+j_2(t)$ along with a substitution in the above calculations.

Now, we turn our attention to the surface integral to continue our derivation of an integral representation of $H^*_{B}H_{B}$. Notice that
\begin{align*}
    \int_{\partial B} \text{e}^{-\text{i}k\hat{x}\cdot z} (\partial_\nu H_{B}g) \,\text{d}s(z) &= \int_{\partial B} \text{e}^{-\text{i}k\hat{x}\cdot z} \Bigg[ \int_{\mathbb{S}^{d-1}} \partial_\nu \text{e}^{\text{i}kz\cdot\hat{y}} g(\hat{y}) \,\text{d}s(\hat{y}) \Bigg] \,\text{d}s(z) \\
    & = \int_{\partial B} \text{e}^{-\text{i}k\hat{x}\cdot z} \Bigg[ \int_{\mathbb{S}^{d-1}} \partial_r \text{e}^{\text{i}kr\hat{z}\cdot\hat{y}} g(\hat{y})\,\text{d}s(\hat{y}) \Bigg] \,\text{d}s(z) \\
    &= \int_{\partial B} \text{e}^{-\text{i}k\hat{x}\cdot z} \Bigg[ \int_{\mathbb{S}^{d-1}} (\text{i}k\hat{z}\cdot\hat{y}) \text{e}^{\text{i}kr \hat{z}\cdot\hat{y}} g(\hat{y})\,\text{d}s(\hat{y}) \Bigg] \,\text{d}s(z) 
\end{align*}   
Here we have used the fact that we are integrating over $\partial B$, which would imply that $\partial_\nu = \partial_r$ when written in polar/spherical coordinates. Continuing with the calculations we see that 
\begin{align*} 
  \int_{\partial B} \text{e}^{-\text{i}k\hat{x}\cdot z} (\partial_\nu H_{B}g) \,\text{d}s(z)  &= \int_{\mathbb{S}^{d-1}} \Bigg[ \int_{\partial B} (\text{i}k\hat{z}\cdot\hat{y})\text{e}^{-\text{i}k {z}\cdot (\hat{x}-\hat{y})} \text{d}s(z) \Bigg]g(\hat{y})\,\text{d}s(\hat{y}) \\
    &= \int_{\mathbb{S}^{d-1}} \Bigg[ \text{e}^{-\text{i}kz_j\cdot (\hat{x}-\hat{y})}\int_{\partial B(0,\epsilon)} (\text{i}k\hat{z}\cdot\hat{y})\text{e}^{-\text{i}kz\cdot (\hat{x}-\hat{y})} \text{d}s(z) \Bigg]g(\hat{y})\,\text{d}s(\hat{y}) \\
    &= \int_{\mathbb{S}^{d-1}} \Bigg[ \text{e}^{-\text{i}kz_j\cdot (\hat{x}-\hat{y})}\text{i}k\hat{y}\epsilon^{d-1}\cdot\int_{\mathbb{S}^{d-1}} \hat{z}\text{e}^{-\text{i}k\hat{z}\cdot \epsilon(\hat{x}-\hat{y})} \,\text{d}s(\hat{z}) \Bigg]g(\hat{y})\,\text{d}s(\hat{y})
\end{align*}
We now use the following Funk–Hecke integral identities(see for e.g. \cite{nf-harris})
\begin{align*}
     \int_{\mathbb{S}^{d-1}} \hat{z}\text{e}^{-\text{i}k\hat{z}\cdot r(x-y)} \, \text{d}s(\hat{z}) = \begin{cases}
        2\pi \frac{(x-y)}{\text{i}|x-y|}J_1(kr|x-y|) & \quad \text{for } d=2, \\ \\
        4\pi \frac{(x-y)}{\text{i}|x-y|}j_1(kr|x-y|) & \quad \text{for } d=3
    \end{cases}
\end{align*}
to get the following expression
\begin{align*}
     \int_{\partial B} (\text{i}k\hat{z}\cdot\hat{y})\text{e}^{-\text{i}kz\cdot (\hat{x}-\hat{y})} \,\text{d}s(z) &= 2\pi \text{e}^{-\text{i}kz_j\cdot (\hat{x}-\hat{y})} \text{i} k \epsilon \hat{y} \cdot  \frac{(\hat{x}-\hat{y})}{\text{i}|\hat{x}-\hat{y}|}J_1(k\epsilon|\hat{x}-\hat{y}|) \\
     &=  \text{e}^{-\text{i}kz_j\cdot (\hat{x}-\hat{y})}\pi k^2\epsilon^2 (\hat{x}\cdot\hat{y}-1) \Bigg( J_0(k\epsilon|\hat{x}-\hat{y}|) +  J_2(k\epsilon|\hat{x}-\hat{y}|) \Bigg)
\end{align*}
for $d=2$. Notice, that we have used that $\hat{y} \cdot (\hat{x}-\hat{y}) = \hat{x}\cdot\hat{y}-1$ to now get the expression
\begin{align*}
    &\int_{\partial B} \text{e}^{-\text{i}k\hat{x}\cdot z} (\partial_\nu H_{_B}g) \, \text{d}s(z) \\
    &\hspace{0.5in}= \int_{\mathbb{S}^{1}} \Bigg[  \text{e}^{-\text{i}kz_j\cdot (\hat{x}-\hat{y})}\pi k^2\epsilon^2 (\hat{x}\cdot\hat{y}-1) \Bigg( J_0(k\epsilon|\hat{x}-\hat{y}|) +  J_2(k\epsilon|\hat{x}-\hat{y}|) \Bigg) \Bigg]g(\hat{y})\, \text{d}s(\hat{y}) .
\end{align*}
Again, similar calculations for $d=3$, we have that
\begin{align*}
    &\int_{\partial B} \text{e}^{-\text{i}k\hat{x}\cdot z} \partial_\nu H_{B}g \, \text{d}s(z) \\
    &\hspace{0.5in}= \int_{\mathbb{S}^{2}} \Bigg[  \text{e}^{-\text{i}kz_j\cdot (\hat{x}-\hat{y})}\frac{4}{3}\pi k^2\epsilon^3 (\hat{x}\cdot\hat{y}-1) \Bigg( j_0(k\epsilon|\hat{x}-\hat{y}|) +  j_2(k\epsilon|\hat{x}-\hat{y}|) \Bigg) \Bigg]g(\hat{y})\, \text{d}s(\hat{y}) .
\end{align*}
With this, we can get the integral representation for $H^*_{B}H_{B}$, given by
\begin{align*}
    (H^*_{B}H_{B})g(\hat{x}) =  \int_{\mathbb{S}^{1}} \Bigg[  \text{e}^{-\text{i}kz_j\cdot (\hat{x}-\hat{y})} \pi\epsilon^2 (1+k^2\hat{x}\cdot\hat{y}) \Bigg( J_0(k\epsilon|\hat{x}-\hat{y}|) +  J_2(k\epsilon|\hat{x}-\hat{y}|) \Bigg) \Bigg]g(\hat{y})\, \text{d}s(\hat{y}) 
\end{align*}
for $d=2$. Similarly, by combining our calculations we have 
\begin{align*}
    (H^*_{B}H_{B})g(\hat{x})  = \int_{\mathbb{S}^{2}} \Bigg[  \text{e}^{-\text{i}kz_j\cdot (\hat{x}-\hat{y})}\frac{4}{3}\pi \epsilon^3 (1+k^2\hat{x}\cdot\hat{y}) \Bigg( j_0(k\epsilon|\hat{x}-\hat{y}|) +  j_2(k\epsilon|\hat{x}-\hat{y}|) \Bigg) \Bigg]g(\hat{y})\, \text{d}s(\hat{y}) 
\end{align*}
for $d=3$. 

Notice, that by our calculations we obtained that  
$$(H^*_{B}H_{B})g(\hat{x}) = \int_{\mathbb{S}^{d-1}} h(\hat{x} , \hat{y}) g(\hat{y})\, \text{d}s(\hat{y})$$ 
where the kernel function is given by 
$$h(\hat{x} , \hat{y})= \begin{cases}
         \text{e}^{-\text{i}kz_j\cdot (\hat{x}-\hat{y})} \pi\epsilon^2 (1+k^2\hat{x}\cdot\hat{y}) \Bigg( J_0(k\epsilon|\hat{x}-\hat{y}|) +  J_2(k\epsilon|\hat{x}-\hat{y}|) \Bigg) & \quad \text{for } d=2, \\ \\
         \text{e}^{-\text{i}kz_j\cdot (\hat{x}-\hat{y})}\frac{4}{3}\pi \epsilon^3 (1+k^2\hat{x}\cdot\hat{y}) \Bigg( j_0(k\epsilon|\hat{x}-\hat{y}|) +  j_2(k\epsilon|\hat{x}-\hat{y}|) \Bigg) & \quad \text{for } d=3
    \end{cases}$$
where we see that $h(\hat{x} , \hat{y}) =\overline{h(\hat{y},\hat{x})}$. This is to be expected due to the fact that $H^*_{B}H_{B}$ is a self-adjoint operator. With this we have all we need to provide some numerical reconstructions using the monotonicity result in Theorem \ref{aniso-mono-method}.

\subsection{Numerical Reconstruction}
Now, we will provide some numerical examples for reconstruction of the scatterer $D$ via the monotonicity method. To this end, we will use separation of variables to obtain an integral representation for the far--field operator $F$. For simplicity, we consider the direct scattering problem with constant coefficients on $B(0,R)$, where assume $A=a I$, with constant $a \in \mathbb{C}$ along with constants $n$, and $\eta$. Therefore, we have that \eqref{directacbc1}--\eqref{directacbc2} is given by 
\begin{align*}
    \Delta u^s + k^2u^s = 0 \quad \text{in } \mathbb{R}^d\setminus \overline{B(0,R)} \quad \text{ and }\quad \Delta u + \frac{n}{a}k^2u = 0 \quad &\text{in } B(0,R) \\
    (u^s+u^i)^+ = u^- \quad \text{and} \quad a\partial_r u^- = \partial_r(u^s+u^i)^+ + \eta(u^s+u^i) \quad &\text{on } \partial B(0,R)
\end{align*}
along with the radiation condition as $|x| \to \infty$ for the scattered field $u^s$. Recall that we can express the incident field $u^i=\text{e}^{\text{i}kx\cdot \hat{y}}$, using the Jacobi--Anger expansion is given by 
$$u^i(r,\theta) = \sum_{p\in \mathbb{Z}} \text{i}^p J_p(kr)\text{e}^{\text{i}p(\theta - \phi)}$$
where $x= r \big(\cos(\theta),  \sin(\theta) \big)$ and $\hat{y} = \big(\cos(\phi), \sin(\phi) \big)$.

To solve for the scattered field $u^s$, and the total field $u$ for this problem we make the assumption that they can be written as
$$u^s(r,\theta) = \sum_{p\in \mathbb{Z}} \text{i}^p \text{u}^s_p H_p^{(1)}(kr)\text{e}^{\text{i}p(\theta - \phi)} \quad \text{and} \quad u(r,\theta)= \sum_{p\in \mathbb{Z}} \text{i}^p \text{u}_p J_p\Big( k\sqrt{\frac{n}{a}}r\Big)\text{e}^{\text{i}p(\theta - \phi)}$$
where $H^{(1)}_p$ are Hankel functions of the first kind of order $p$. Here $\text{u}^s_p$ and $\text{u}_p$ can be seen as the coefficients that are to be determined by using the boundary conditions at $r=R$. Therefore, we have that $\text{u}^s_p$ and $\text{u}_p$ satisfy the $2 \times 2$ linear system  
\begin{align*}
     \begin{pmatrix}
        H_p^{(1)}(kR) & -J_p\Big( k\sqrt{\frac{n}{a}}R\Big) \\
        kH_p^{(1)'}(kR)+\eta H_p^{(1)}(kR) & -k\sqrt{na}J'_p\Big( k\sqrt{\frac{n}{a}}R\Big)
    \end{pmatrix} 
    \begin{pmatrix}
        \text{u}^s_p \\
        \text{u}_p
    \end{pmatrix} 
    =
    \begin{pmatrix}
        -J_p(kR) \\
        -kJ'_p(kR) - \eta J_p(kR)
    \end{pmatrix}
\end{align*}
for each $p \in \Z$. We can solve for the coefficients $\text{u}^s_p$ by using Cramer's Rule, we have
\begin{align*}
    \text{u}^s_p = \frac{\text{Det}(M_x)}{\text{Det}(M)} \quad \text{ with matrices } \quad &M=
    \begin{pmatrix}
        H_p^{(1)}(kR) & -J_p\Big( k\sqrt{\frac{n}{a}}R\Big) \\
        kH_p^{(1)'}(kR)+\eta H_p^{(1)}(kR) & -k\sqrt{na}J'_p\Big( k\sqrt{\frac{n}{a}}R\Big)
    \end{pmatrix} \\
    \text{ and } \quad &M_x=
    \begin{pmatrix}
        -J_p(kR) & -J_p\Big( k\sqrt{\frac{n}{a}}R\Big) \\
        -kJ'_p(kR) - \eta J_p(kR) & -k\sqrt{na}J'_p\Big( k\sqrt{\frac{n}{a}}R\Big)
    \end{pmatrix}.
\end{align*}
With this computation of $\text{u}^s_p$, we get an expression for the far--field pattern is given by
\begin{align*}
    u^\infty(\hat{x},\hat{y}) \approx  \frac{4}{\text{i}} \sum_{|p| = 0}^{15} \text{u}^s_p \text{e}^{\text{i}p(\theta - \phi)}\quad \text{ where } \quad \text{u}^s_p = \frac{\text{Det}(M_x)}{\text{Det}(M)}. 
\end{align*}
In all our numerical experiments we use the above approximation of the far-field pattern.

In order to numerically compute the eigenvalues of $\pm \text{Re}(F)- H^*_{B}H_{B}$, we approximate the integral representations of $F$ and $H^*_{B}H_{B}$ with a 64--point Riemann sum. The $64 \times 64$ discretization of the operators $F$ and $H^*_{B}H_{B}$ are computed as 
\begin{align*}
    \textbf{F} = \left[ \frac{2\pi}{64} u^\infty (\hat{x}_i ,\hat{y}_j ) \right]_{i,j=1}^{64} \quad \text{ and } \quad \textbf{H}^*_B\textbf{H}_B = \left[ \frac{2\pi}{64} h(\hat{x}_i ,\hat{y}_j ) \right]_{i,j=1}^{64} \quad \text{ respectively}.
\end{align*}
We recall, that the kernel function to define $\textbf{H}^*_B\textbf{H}_B$ is given by 
$$ h(\hat{x} ,\hat{y} ) = \text{e}^{-\text{i}kz_j\cdot (\hat{x}-\hat{y})} \pi\epsilon^2 (1+k^2\hat{x}\cdot\hat{y}) \Bigg( J_0(k\epsilon|\hat{x}-\hat{y}|) +  J_2(k\epsilon|\hat{x}-\hat{y}|) \Bigg) .$$
Here we discretize the $\mathbb{S}^1$ such that 
$$\hat{x}_i=\hat{y}_i=(\cos(\theta_i),\sin(\theta_i))\quad \text{ where } \quad \theta_i=2\pi(i-1)/64 \,\, \text{ for} \,\,i=1,\dots,64.$$ 
In order to test the stability we will add random noise to the discretized far--field operator $\mathbf{F}$ such that
$$\mathbf{F}_{\delta}=\Big[ \mathbf{F}_{i,j}(1+\delta\mathbf{E}_{i,j})\Big]^{64}_{i,j=1}\hspace{.5cm}\text{where}\hspace{.5cm} \| \mathbf{E}\|_2=1.$$
Here, the matrix $\mathbf{E} \in \C^{64 \times 64}$ is taken to have random entries and $0<\delta < 1$ is the relative noise level added to the data.

In our numerical computations, we consider recovering $D=B(0,R)$ for different values of $R \in (0,1]$. Therefore, begin by discretizing a rectangular region $\Omega = [-2,2] \times [-2,2]$ that completely encompasses our region $D=B(0,R)$, and we fix a small radius $\epsilon$ of our ball $B$. In all our numerical examples we pick the radius $\epsilon = 10^{-4}$.
According to the theory, as we move our ball $B = B(z,\epsilon)$ along the grid $\Omega$, we should have that there are more positive eigenvalues of the self-adjoint matrix 
$$\pm \text{Re}(\textbf{F}_{\delta}) - \textbf{H}^*_B\textbf{H}_B \quad \text{ where } \quad \text{Re}(\textbf{F}_{\delta}) = \frac{1}{2} \big(\textbf{F}_{\delta} + \textbf{F}_{\delta}^* \big)$$ 
provided that $z \in D$ and less positive eigenvalues if $z \not\in D$, roughly speaking. Therefore, at each point in the region $\Omega$, we compute the imaging function 
$$W(z) = \#\Big\{ \lambda_j \in  \text{eig} \big( \pm \mathrm{Re}(\textbf{F}_{\delta})-\textbf{H}^*_B\textbf{H}_B \big) \, \, : \, \, \lambda_j >0 \Big\}.$$
This is just the cardinality of the set of positive eigenvalues. To evaluate the imaging function we use the built in \texttt{eig} command in \texttt{MATLAB}. This gives us a contour plot that will serve as our reconstruction for $D$. We remark that in our experiments we do not attempt to regularize the imaging function(see for e.g. \cite{mono-w-reg}). This is due to the simplicity of the geometry we are considering. \\

\noindent{\bf Example 1:} For this reconstruction, we will provide two examples with minimal amounts of noise in the data. We provide reconstructions with added $1\%$ noise i.e. $\delta =0.01$ in our calculations. The reconstructions are shown in Figure \ref{acbc_reconst_1noise}.  \\
    \begin{figure}[H]
        \centering 
        \includegraphics[scale=0.59]{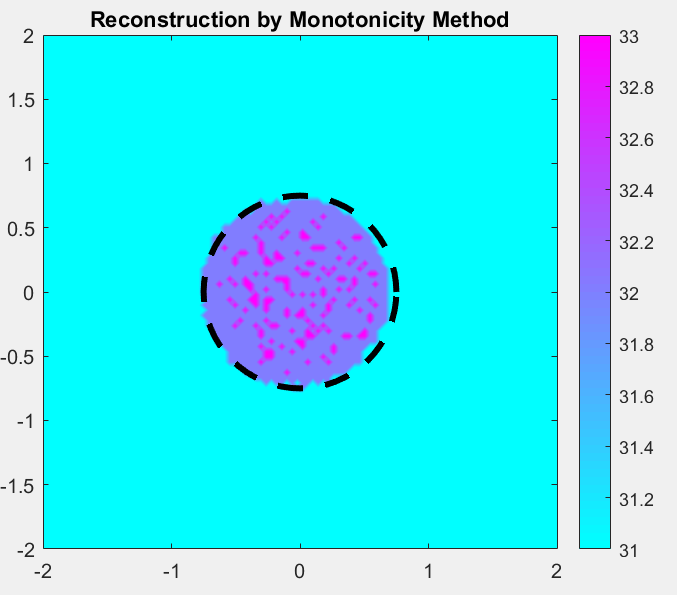} \hspace{0.2in} \includegraphics[scale=.87]{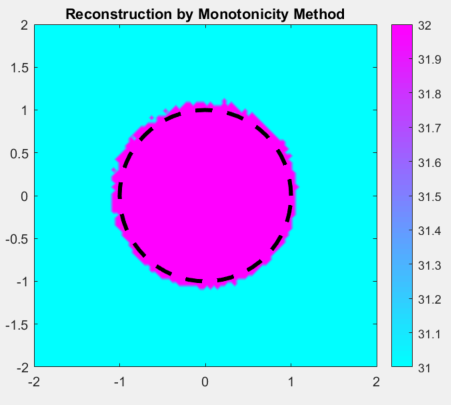}
        \caption{ Reconstruction of domains with $1\%$ noise. Here $D=B(0,3/4)$ with $k = 10$, $a = 2-\text{i}$, $n = 1/2$, $\eta = 2$ (left) and $D=B(0,1)$ with $k = 2\pi$, $a = 1/4$, $n = 2$, $\eta = 2$ (right).}
        \label{acbc_reconst_1noise}
    \end{figure}

\noindent{\bf Example 2:} We provide reconstruction examples for the same scatterers in the previous example with more noise added added to the data i.e. $5\%$ noise. The reconstructions are shown in Figure \ref{acbc_reconst_5noise}.\\
    \begin{figure}[H]
        \centering 
        \includegraphics[scale=0.6]{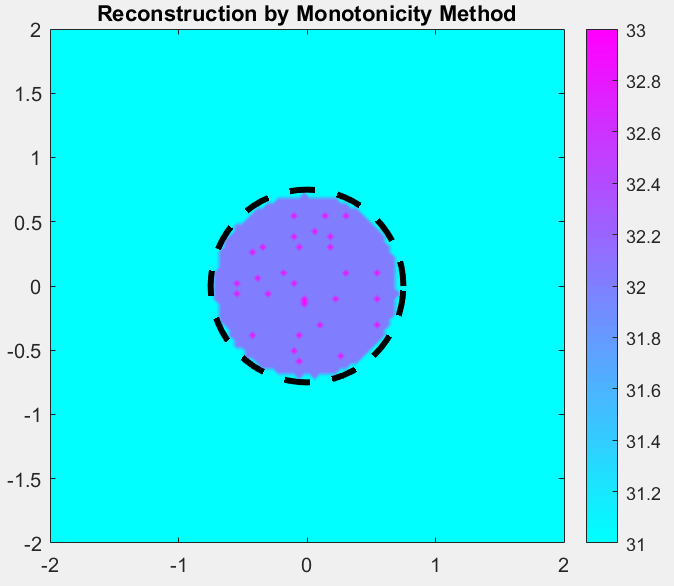} \hspace{0.2in} \includegraphics[scale=.58]{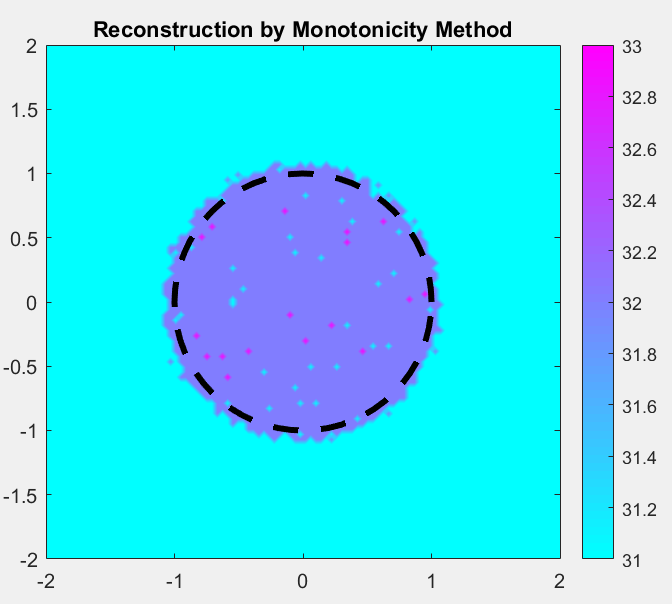}
        \caption{ Reconstruction of domains with $5\%$ noise. Here $D=B(0,3/4)$ with $k = 10$, $a = 2-\text{i}$, $n = 1/2$, $\eta = 2$ (left) and $D=B(0,1)$ with $k = 2\pi$, $a = 1/4$, $n = 2$, $\eta = 2$ (right).}
        \label{acbc_reconst_5noise}
    \end{figure}

\noindent{\bf Example 3:} Now, we provide an example with added $10\%$ noise with slightly different refractive indices $n$ in the scatterer. The reconstructions are shown in Figure \ref{acbc_reconst_10noise}. 
    \begin{figure}[H]
        \centering 
        \includegraphics[scale=0.6]{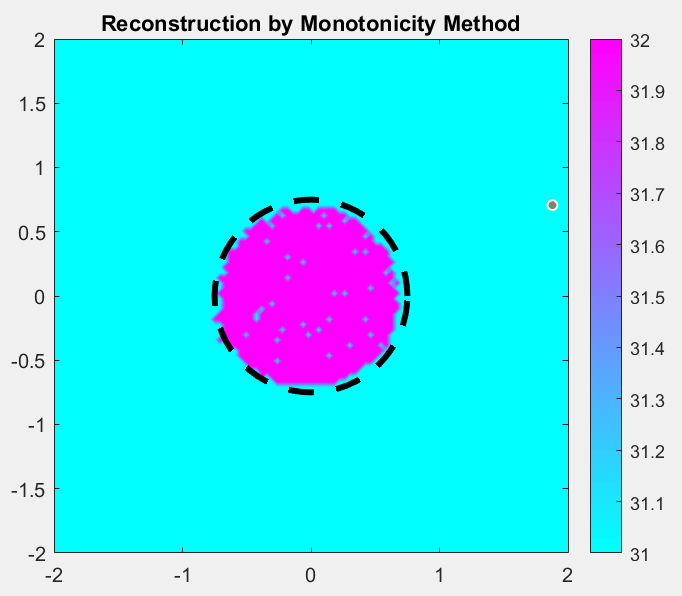} \hspace{0.2in} \includegraphics[scale=.58]{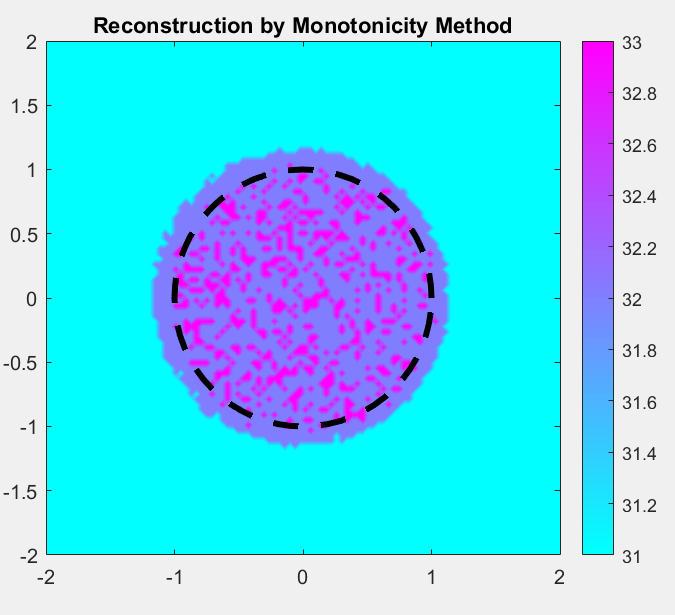}
        \caption{ Reconstruction of domains with $10\%$ noise. Here $D=B(0,3/4)$ with $k = 10$, $a = 2-\text{i}$, $n = 1/2$, $\eta = 2$ (left) and $D=B(0,1)$ with $k = 2\pi$, $a = 1/4$, $n = 2$, $\eta = 2$ (right).}
        \label{acbc_reconst_10noise}
    \end{figure}

\noindent{\bf Example 4:} Lastly, we provide an example where there is no conductivity parameter $\eta$. Also, we are inspired by \cite{mono_AG} and consider finding the radius of the region $D$ by plotting the number of positive eigenvalues where the sampling point is fixed but the sampling ball $B(0 , \epsilon)$ has radius $\epsilon \in (0,2]$. As we see in Figure \ref{acbc_reconst_radiuseg}, the number of positive eigenvalues of our operator decrease as the sampling radius increases. Here, the radius $R=1/2$ for the scatterer $D$ and the number of positive eigenvalues decreases almost monotonically when the radius of the sampling ball $\epsilon>1/2$. This is consistent with the theoretical results. 

   \begin{figure}[H]
        \centering 
        \includegraphics[scale=0.2]{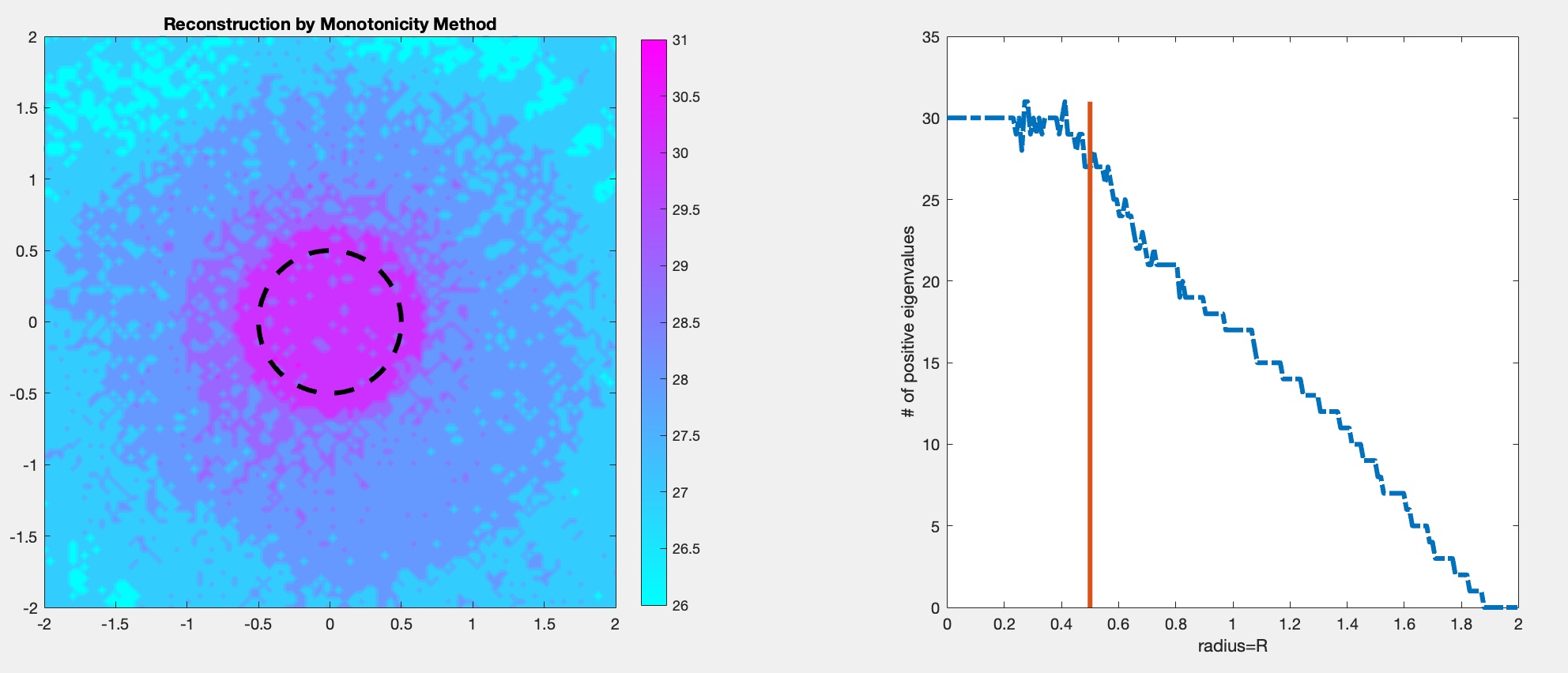}
        \caption{ Reconstruction of domains with $2\%$ noise. Here $D=B(0,1/2)$ with $k = 10$, $a = 4$, $n = 1/2$, $\eta = 0$ (left) and the plot of the number of positive eigenvalues as a function of $\epsilon$ where the sampling ball is $B(0 , \epsilon )$ for $\epsilon \in (0,2]$ (right).}
        \label{acbc_reconst_radiuseg}
    \end{figure} 

From our examples, we have faithful reconstructions for the circular regions under either set of assumptions on the coefficients. We also see that the method is resilient to added noise in these examples. With this we see that the monotonicity method is applicable to reconstructing anisotropic scatterers with or without a conductive boundary condition. 


\section{Conclusions} \label{acbc_section_conclusion}
In this paper, we studied the monotonicity method for an anisotropic material with a conductive boundary on an unbounded domain. We obtained the necessary symmetric factorization of the far--field to apply the monotonicity method. Our main contribution was adopting the theory of the monotonicity method for the anisotropic material under two different assumptions on its physical parameters. We then provided some numerical examples of using the method on different circular domains in two dimensions with various levels of noise. We see that this method creates faithful reconstructions that are resilient to noise. With that, there is still room for more extensive numerical tests of the monotonicity method studied in this paper. Another area to be investigated is applying this analysis for near--field data \cite{near_field} as well as scattering in waveguides \cite{mono-waveguide,fm-waveguide} and by periodic structures \cite{periodictevs} since the factorization of the data operators has been studied. \\

\noindent{\bf Acknowledgments:} The research of V. Hughes, I. Harris and H. Lee is partially supported by the NSF DMS Grant 2107891.



\begin{thebibliography}{99}
\bibitem{mono_AG}
\newblock A. Albicker and R. Griesmaier,
\newblock Monotonicity in inverse obstacle scattering on unbounded domains, 
\newblock {\it Inverse Problems}, {\bf 36} (2020), 085014.


\bibitem{mono-waveguide}
T. Arens, R. Griesmaier and R. Zhang,
Monotonicity-based shape reconstruction for an inverse scattering problem in a waveguide
\newblock {\it Inverse Problems}, {\bf 36} (2023), 075009.


\bibitem{mono-p-laplace}
\newblock T. Brander, B. Harrach, M. Kar, and M. Salo, 
\newblock Monotonicity and enclosure methods for the $p$-Laplace equation,
\newblock {\it SIAM J. Appl. Math.}, {\bf 78} (2018), 742--758.


\bibitem{fmconductbc}
\newblock O. Bondarenko and X. Liu,
\newblock The factorization method for inverse obstacle scattering with conductive boundary condition,
\newblock {\it Inverse Problems}, {\bf 29} (2013), 095021.


\bibitem{fm-anisocbc}
\newblock O. Bondarenko and A. Kirsch,
\newblock The factorization method for inverse scattering by a penetrable anisotropic obstacle with conductive transmission conditions,
\newblock {\it Inverse Problems}, {\bf 32} (2016), 105011.


\bibitem{te-cbc}
\newblock O. Bondarenko, I. Harris, and A. Kleefeld, 
\newblock The interior transmission eigenvalue problem for an inhomogeneous media with a conductive boundary, 
\newblock {\it Applicable Analysis}, {\bf 96(1)} (2017), 2--22.


\bibitem{fm-waveguide}
\newblock L. Borcea and S. Meng, 
\newblock Factorization method versus migration imaging in a waveguide, 
\newblock {\it  Inverse Problems}, {\bf 35} (2019), 124006.


\bibitem{mono-eit-xtreme}
\newblock V. Candiani, J. Darde, H. Garde, and N. Hyvonen, 
\newblock Monotonicity-based reconstruction of extreme inclusions in electrical impedance tomography, 
\newblock {\it SIAM J. Math. Anal.}, {\bf 52} (2020), 6234--6259.



\bibitem{Cakoni-Colton-book} 
\newblock F. Cakoni, D. Colton,
\newblock {\it``A Qualitative Approach to Inverse Scattering Theory''},
\newblock {Springer}, {New York},  (2013).


\bibitem{CCH-book} 
\newblock F. Cakoni, D. Colton, and H. Haddar,
\newblock {\it``Inverse Scattering Theory and Transmission Eigenvalues''},
\newblock {CBMS Series, SIAM 88}, {Philadelphia},  (2016).


\bibitem{LSM-maxwell-book}
\newblock F. Cakoni, D. Colton, and P. Monk,
\newblock {\it ``The linear Sampling Method in Inverse Electromagnetic Scattering''}, 
\newblock CBMS Series, SIAM Publications 80, (2011).


\bibitem{On-the-interior-TE} 
\newblock F. Cakoni and A. Kirsch,
\newblock On the interior transmission eigenvalue problem,
\newblock  { \it Int. Jour. Comp. Sci. Math.}, {\bf 3} (2010), 142--167.


\bibitem{Heejin1}
\newblock F. Cakoni, H. Lee,  P. Monk and Y. Zhang,
\newblock{A spectral target signature for thin surfaces with higher order jump conditions},
\newblock{\it Inverse Problems and Imaging}, {\bf 16(6)} (2022), 1473--1500.


\bibitem{fm-shixu}
\newblock F. Cakoni, S. Meng and H. Haddar,
\newblock The factorization method for a cavity in an inhomogeneous medium, 
\newblock {\it Inverse Problems}, {\bf 30} (2014), 045008.


\bibitem{Haddar1}
\newblock S. Chaabane, B. Charfi and H. Haddar, 
\newblock Reconstruction of discontinuous parameters in a second order impedance boundary operator. 
\newblock{\it Inverse Problems} \textbf{32(10)} (2016), 105004.


\bibitem{MUSIC-LSM}
\newblock M. Cheney, 
\newblock The linear sampling method and the MUSIC algorithm, 
\newblock {\it Inverse Problems},  {\bf 17} (2001), 591--595.


\bibitem{mono-invscat-crack}
\newblock T. Daimon, T. Furuya and R. Saiin,
\newblock The monotonicity method for the inverse crack scattering problem,
\newblock {\it Inverse Problems in Sci. and Eng.}, {\bf 28} (2020), 1570--1581.


\bibitem{mono_method_paper}
\newblock T. Furuya,
\newblock Remarks on the factorization and monotonicity method for inverse acoustic scatterings, 
\newblock {\it Inverse Problems}, {\bf 37} (2021), 065006.


\bibitem{mono-w-reg}
\newblock H. Garde and S. Staboulis, 
\newblock Convergence and regularization for monotonicity-based shape reconstruction in electrical impedance tomography, 
\newblock {\it Numer.Math.}, {\bf 135} (2017), 1221--1251.


\bibitem{mono-invscat-media}
\newblock R. Griesmaier and B. Harrach,
\newblock Monotonicity in Inverse Medium Scattering on Unbounded Domains,
\newblock {\it SIAM J. Appl. Math.}, {\bf 78} (2018), 2533--2557.


\bibitem{GyCo}
\newblock F. Gylys-Colwell, 
\newblock An inverse problem for the Helmholtz equation, 
\newblock {\it Inverse Problems} {\bf 12} (1996), 139-156. 


\bibitem{mono-schrodinger}
\newblock B. Harrach and Y-H. Lin,
\newblock Monotonicity-based inversion of the fractional Schrodinger equation I. Positive potentials,
\newblock {\it SIAM J. Math. Anal.}, {\bf 51} (2019), 3092--3111.


\bibitem{mono-eit}
\newblock B. Harrach and M. Ullrich, 
\newblock Monotonicity-based shape reconstruction in electrical impedance tomography,
\newblock {\it SIAM J. Math. Anal.}, {\bf 45} (2013), 3382--3403.


\bibitem{nf-harris} 
\newblock  I. Harris,
\newblock Direct methods for recovering sound soft scatterers from point source measurements,
\newblock {\it Computation},  {\bf 9(11)} (2021), 120.


\bibitem{regfm2}
\newblock I. Harris, 
\newblock Regularized factorization method for a perturbed positive compact operator applied to inverse scattering,
\newblock {\it  Inverse Problems}, {\bf 39} (2023),  115007.


\bibitem{periodictevs}
\newblock  I. Harris, D.-L. Nguyen, J. Sands, and T. Truong,
\newblock  On the inverse scattering from anisotropic periodic layers and transmission eigenvalues. 
\newblock  {\it Applicable Analysis} {\bf 101(8)} (2022), 3065--3081.


\bibitem{near_field}
\newblock I. Harris and S. Rome,
\newblock Near field imaging of small isotropic and extended anisotropic scatterers,
\newblock {\it Applicable Analysis}, {\bf 96(10)}  (2017) 1713--1736.


\bibitem{aniso-te-cbc}
\newblock V. Hughes, I. Harris, and J. Sun, 
\newblock The anisotropic transmission eigenvalue problem with a conductive boundary. (arXiv:2311.00526) 


\bibitem{factor-music}
\newblock A. Kirsch, The MUSIC algorithm and the factorization method in inverse scattering theory for inhomogeneous media, 
\newblock {\it Inverse Problems} {\bf 18 } (2002), 1025--1040.


\bibitem{kirschbook}
\newblock A. Kirsch and N. Grinberg, 
\newblock {\it ``The Factorization Method for Inverse Problems''.}
\newblock 1st edition Oxford University Press, Oxford 2008.


\bibitem{fm-anisoKL}
\newblock A. Kirsch and X. Liu,
\newblock The Factorization Method for Inverse Acoustic Scattering by a Penetrable Anisotropic Obstacle. 
\newblock {\it Math. Methods in the App.Sci.} {\bf 37(8)} (2014), 1159--1170. 




\bibitem{salsa} 
\newblock S. Salsa,
\newblock {\it``Partial Differential Equations in Action From Modelling to Theory''},
\newblock Springer Italia, Milano, (2008).


\bibitem{mono-other}
\newblock A. Tamburrino, 
\newblock Monotonicity based imaging methods for elliptic and parabolic inverse problems 
\newblock {\it J. Inverse Ill-Posed Probl.}, {\bf 14} (2006), 633--642.


\bibitem{otheraniso-tecbc}
\newblock J. Xiang and G. Yan,
\newblock The interior transmission eigenvalue problem for an anisotropic medium by a partially coated boundary,
\newblock  {\it Acta Mathematica Scientia}  {\bf 44}  (2024), 339--354.
 
\end{thebibliography}
\end{document}